\documentclass[11pt]{article}
\usepackage[a4paper,margin=26mm]{geometry}
\usepackage[T1]{fontenc}
\usepackage{lmodern}
\usepackage{amsmath,amssymb,amsthm,mathtools}
\usepackage{microtype}
\usepackage{enumitem}
\usepackage[hidelinks]{hyperref}

\newtheorem{theorem}{Theorem}[section]
\newtheorem{lemma}[theorem]{Lemma}
\newtheorem{proposition}[theorem]{Proposition}
\newtheorem{corollary}[theorem]{Corollary}
\theoremstyle{remark}
\newtheorem{remark}[theorem]{Remark}
\newcommand{\A}{\mathcal A}
\newcommand{\B}{\mathcal B}
\newcommand{\T}{\mathcal T}
\newcommand{\R}{\mathbb R}
\newcommand{\ind}{\mathbf 1}
\DeclareMathOperator{\osc}{osc}
\DeclareMathOperator{\Var}{Var}
\DeclareMathOperator{\Sym}{Sym}
\setlist[enumerate]{label=(\roman*),nosep,leftmargin=*}
\hypersetup{pdftitle={Submodular decompositions: a bounded-variation counterexample and an exact planar cut formula}}

\title{\Large\bfseries Submodular decompositions: a bounded-variation\\[0.25em]
counterexample and an exact planar cut formula}
\author{%
\normalsize\textsc{Guoqing Wang}\quad
\textsc{Linlin Wang}\quad
\textsc{Jun Yue}}
\date{}

\begin{document}
\maketitle

\begin{abstract}

B\'erczi, Geh\'er, Imolay, Lov\'asz, and Schwarcz
(\emph{SIAM J. Discrete Math.} \textbf{40} (2026), 308--335)
asked whether every set function of bounded variation is a difference
of two bounded submodular functions, and how to determine the optimum
in a monotonic sum-decomposition of a planar graph cut function.
In this paper, we answer both questions and obtain quantitative
extensions. For the first question, we characterize the continuous
profiles $f:[0,1]\to\R$ for which $A\mapsto f(m(A))$ is a bounded
submodular difference on an atomless probability space.
More precisely, we express the minimum oscillation of the subtracted
component through the positive part of the distributional second
derivative. We also characterize admissible discontinuous profiles
and give an exact summability criterion for countable block functions.
In particular, the first question has a negative answer even for
increasing functions of variation one whose scalar profiles have
log-Lipschitz derivatives and belong to every $C^{1,\alpha}$,
$0<\alpha<1$. Turning to graph cuts, we prove that the optimum equals
the total edge weight minus the maximum fractional triangle-packing
weight for every nonnegative weighting if and only if the graph
is $K_5$-free. Moreover, maximum packings parametrize the symmetrized
optimal components, while on $K_4$-free graphs we describe all optimal
components and characterize uniqueness. Consequently, the formula
applies to weighted planar graphs and implies edge-weight monotonicity
on $K_5$-free graphs. For rational weights, it also yields
polynomial-time constructions.

\end{abstract}

\begingroup

\endgroup

\section{Introduction}

Let $\A$ be an algebra of subsets of a set $J$. A set function
$\varphi:\A\to\R$ is \emph{submodular} if
$\varphi(X)+\varphi(Y)\ge\varphi(X\cap Y)+\varphi(X\cup Y)$ for
all $X,Y\in\A$. It is \emph{increasing} if
$\varphi(X)\le\varphi(Y)$ whenever $X\subseteq Y$, and
\emph{decreasing} if the reverse inequality holds. A set function
is bounded if its absolute value is uniformly bounded.
A set function $\varphi$ has \emph{bounded variation} if the sums
$\sum_{i=1}^r|\varphi(X_i)-\varphi(X_{i-1})|$ are uniformly bounded
over all finite chains
$\varnothing=X_0\subseteq\cdots\subseteq X_r=J$ in $\A$.
Equivalently, $\varphi$ is a difference of two increasing
real-valued set functions. The components in this characterization
are not required to be submodular. Thus bounded variation provides
a natural starting point for studying the additional restrictions
imposed by submodularity. For $\varphi(\varnothing)=0$, a \emph{monotonic
sum-decomposition} is a representation $\varphi=u+v$,
where $u$ is increasing and submodular, $v$ is
decreasing and submodular, and
$u(\varnothing)=v(\varnothing)=0$.
We also consider representations $\varphi=g-h$ with $g,h$ bounded and submodular,
without a monotonicity requirement. Such a representation is called
a \emph{bounded submodular difference}. Here and below, normalization
of a decomposition refers to the values of its components at the
empty set.

Submodular functions occur in both combinatorial optimization
and the analytic theory of capacities. Lov\'asz~\cite{Lovasz}
connects these settings in the study of infinite ground sets,
while his work on the Choquet extension~\cite{LovaszChoquet}
uses bounded variation to treat integration without monotonicity.
These connections motivate the study of decompositions that preserve
submodularity. On a finite ground set, every submodular function
admits a monotonic sum-decomposition. More generally,
every set function is a difference of two increasing submodular
functions~\cite[Theorem~4.18]{BGILS}, with a constant restored when
its value at the empty set is nonzero. On infinite ground sets,
B\'erczi, Geh\'er, Imolay, Lov\'asz, and
Schwarcz~\cite{BGILS} showed that monotonic
decompositions can fail even for bounded submodular functions.
They also proved existence under weak infinite alternation,
an additional family of higher-order inequalities.

In this paper, we address two questions left open in~\cite{BGILS}.
The first asks whether every set function of bounded variation is
a difference of two bounded submodular functions. The second asks
for the minimum value of the increasing component on the full vertex
set in a monotonic sum-decomposition of a planar graph cut function.
The former is an existence question on general set algebras; the
latter is an optimization question on finite graphs. We give a
negative answer to the first and an exact formula for the second,
together with quantitative criteria and descriptions of the optimal
components.

For the bounded-variation question, let $\B$ be the
Borel $\sigma$-algebra of $[0,1]$ and let $\lambda$ be Lebesgue
measure. For a continuous function $f:[0,1]\to\R$, consider the
set function $\Phi_f(A)=f(\lambda(A))$. Concave profiles give
submodular set functions, so a difference of continuous concave
profiles gives a bounded submodular difference. Our first theorem
proves the converse and uses it to construct an increasing
counterexample of variation one.

\begin{theorem}\label{thm:main}
Let $f:[0,1]\to\R$ be continuous, and define $\Phi_f(A)=f(\lambda(A))$ for $A\in\B$. Then $\Phi_f$ is a difference of two bounded submodular functions on $\B$ if and only if $f$ is a difference of two continuous concave functions on $[0,1]$.

Moreover, there exists an increasing $F\in C^1([0,1])$ with $F(0)=0$, $F(1)=1$, and $1/2\le F'(t)\le3/2$ for every $t\in[0,1]$ such that $\Phi_F(A)=F(\lambda(A))$ is not a difference of two bounded submodular functions. The set function $\Phi_F$ is increasing, has variation one, and satisfies $|\Phi_F(A)-\Phi_F(B)|\le\frac32\lambda(A\triangle B)$ for all $A,B\in\B$.
\end{theorem}

A continuous function on $[0,1]$ is called \emph{DC} if it is a
difference of two continuous convex functions, or equivalently,
of two continuous concave functions; see Hartman~\cite{Hartman}.
Theorem~\ref{thm:main} therefore identifies the admissible continuous
profiles precisely as the DC functions. In the necessity direction,
the bounded submodular components are arbitrary: neither dependence
only on measure nor any continuity or invariance assumption is
required. Nevertheless, they can be replaced by components arising
from continuous concave profiles. This characterization extends
to every atomless probability space, where we can also determine
the minimum oscillation of the subtracted component. To state this
refinement, let $(X,\A,m)$ be an atomless probability space,
let $f\in C([0,1])$, and put $\Phi_f(A)=f(m(A))$. For a bounded set
function $h$, write
$\osc(h)=\sup_{A\in\A}h(A)-\inf_{A\in\A}h(A)$,
and define
\[
 \mathcal C(f)=\inf\{\osc(h):\Phi_f=g-h,\quad
                    g,h\text{ bounded and submodular}\},
\]
with $\mathcal C(f)=+\infty$ if no such decomposition exists.
The next theorem expresses this quantity in terms of the kernel
$K(x,t)=\min\{x,t\}-xt$ and the distributional second derivative
$\mu=f''$ on $(0,1)$. When $\mu$ is a signed locally finite Radon
measure, we denote its positive part by $\mu^+$ and its total
variation measure by $|\mu|$.

\begin{theorem}\label{thm:continuous}
The following are equivalent:
\begin{enumerate}
\item $\Phi_f$ is a difference of two bounded submodular functions;
\item $f$ is a difference of two continuous concave functions on $[0,1]$;
\item $\mu=f''$ is a signed locally finite Radon measure on $(0,1)$ and
\[
 \int_{(0,1)}t(1-t)\,d|\mu|(t)<\infty.
\]
\end{enumerate}
When these conditions hold,
\[
 \mathcal C(f)=
 \max_{x\in[0,1]}\int_{(0,1)}K(x,t)\,d\mu^+(t).
\]
The infimum is attained by components depending only on measure.
\end{theorem}

Thus restricting the components to functions of measure does not
increase the optimal oscillation. The weighted curvature condition
also accounts for possible singular behavior near the endpoints:
the unweighted total variation of $f''$ need not be finite.
This raises the question of whether continuity of the profile is
itself necessary. The next proposition shows that a bounded
submodular difference forces the interior profile to have a
continuous DC extension to $[0,1]$, while leaving the original
endpoint values unrestricted.
\begin{proposition}\label{prop:automatic}
Let $(X,\A,m)$ be an atomless probability space. For an
arbitrary real-valued $f:[0,1]\to\R$, the function $f\circ m$ is a
difference of two bounded submodular functions if and only if
$f|_{(0,1)}$ agrees with the restriction of a continuous DC function
on $[0,1]$. The two endpoint values of $f$ are unrestricted.
\end{proposition}

Having characterized the admissible profiles without assuming
continuity, we next examine how much regularity the counterexamples
can have. Theorem~\ref{thm:smooth} strengthens the construction in
Theorem~\ref{thm:main}: it gives a single increasing profile that
belongs to every $C^{1,\alpha}$, $0<\alpha<1$, has a log-Lipschitz
derivative, and is smooth away from one interior point. The
associated set function still has variation one and admits no
bounded submodular difference. By contrast, every $C^{1,1}$ profile
is DC. The obstruction also has a finite formulation, which allows
us to construct counterexamples on countable ground sets.
A sharp lower bound arising from a single positive second
difference already excludes the profiles $(2t-1)_+^\alpha$,
$0<\alpha<1$, where $x_+=\max\{x,0\}$. The exact finite oscillation
formula in Theorem~\ref{thm:finite} then yields a summability
criterion for functions built from disjoint finite blocks
(Theorem~\ref{thm:blocks}). In particular, there are increasing
functions of variation one on $2^J$, with $J$ countable, that are
continuous in the product topology on $2^J\cong\{0,1\}^J$ but
admit no bounded submodular difference.

We now turn from the existence of bounded submodular differences
to the optimization of monotonic sum-decompositions, as posed in
the planar-graph question. Let $G=(V,E)$ be a finite simple graph
with nonnegative edge weights $w:E\to\R_{\ge0}$. As in~\cite{BGILS},
write $w(F)=\sum_{e\in F}w(e)$ for $F\subseteq E$, and let
$E[S,V\setminus S]$ denote the edges joining $S$ to its complement.
The weighted cut function is $d_w(S)=w(E[S,V\setminus S])$ for
$S\subseteq V$; it is submodular and satisfies
$d_w(S)=d_w(V\setminus S)$. We write $w_e=w(e)$ and $W=w(E)$,
and define its optimal monotonic sum-decomposition value by
\[
 N_G(w)=\|d_w\|_+
 =\min\{u(V):d_w=u+v
       \text{ is a monotonic sum-decomposition}\}.
\]
Here $\|d_w\|_+$ is the notation used in~\cite{BGILS}, and attainment
of the minimum is verified in Section~\ref{sec:prelim}. For unit
weights, we write $d_G=d_{\mathbf 1}$ and $N_G=N_G(\mathbf 1)$.
The edge-by-edge decomposition gives $N_G(w)\le W$, whereas every
admissible pair satisfies $d_w(S)\le u(S)\le u(V)$. Hence
$\max_{S\subseteq V}d_w(S)\le N_G(w)\le W$. Fractional triangle
packings provide a systematic improvement of the upper bound.
Let $\T=\T(G)$ be the family of vertex sets of triangles, with
each edge identified with its two endpoints. A \emph{fractional
triangle packing} is a family $(p_T)_{T\in\T}$ of nonnegative
numbers satisfying $\sum_{T\supseteq e}p_T\le w_e$ for every
$e\in E$. Its maximum total weight is
\begin{equation}\label{eq:packing}
 \nu_w^*(G)=\max\left\{\sum_{T\in\T}p_T:
 p_T\ge0\ (T\in\T),\quad
 \sum_{T\supseteq e}p_T\le w_e\ (e\in E)\right\}.
\end{equation}
The feasible region is nonempty and compact, and the optimum is
zero when $\T$ is empty. In the unweighted case, we write
$\nu^*(G)=\nu_{\mathbf 1}^*(G)$. B\'erczi et al. proved that
$N_G(w)=W$ for triangle-free graphs and, more generally, that
$N_G(w)\le W-\nu_w^*(G)$
\cite[Theorem~7.4 and Corollary~7.8]{BGILS}. They also determined
the optimum for certain unweighted planar graphs
\cite[Examples~7.9--7.10]{BGILS}, leaving the general planar case
open. Our next theorem resolves this question by characterizing
all graphs for which the packing bound is attained for every
nonnegative weighting. As usual, $K_r$-free means containing no
complete subgraph on $r$ vertices.
\begin{theorem}\label{thm:graphs}
Let $G=(V,E)$ be a finite simple graph. The identity
\begin{equation}\label{eq:graphformula}
 N_G(w)=W-\nu_w^*(G)
\end{equation}
holds for every $w\in\R_{\ge0}^E$ if and only if $G$ contains no
$K_5$ subgraph.

Moreover, if $G$ is $K_5$-free and $w$ is a fixed nonnegative
weighting, then maximum fractional triangle packings are in
bijection with the increasing submodular functions $f:2^V\to\R$
satisfying $f(\varnothing)=0$, $f(V)=N_G(w)$, and
$d_w(S)=f(S)+f(V\setminus S)-f(V)$ for every $S\subseteq V$.
Every optimal monotonic sum-decomposition can be symmetrized
to one of these functions without changing its value.
\end{theorem}

Since every planar graph is $K_5$-free, Theorem~\ref{thm:graphs}
answers the planar-graph question, including its weighted version.
The necessity of the graph condition is already witnessed by the
unweighted $K_5$, for which $N_{K_5}=6$ but
$W-\nu^*(K_5)=20/3$. Assigning unit weights to a $K_5$ subgraph
and zero weights elsewhere extends this obstruction to every
graph containing $K_5$. On the other hand, for $K_5$-free graphs,
the formula yields both computational and structural consequences.
For rational weights, an optimal packing and its associated
decomposition can be constructed in polynomial time. Linear
programming duality also supplies optimality certificates and
gives $N_G(w')\le N_G(w)$ whenever $0\le w'\le w$, thereby
resolving the edge-weight monotonicity question
of~\cite[Section~8, item~1]{BGILS} on this class. To describe the
optimal decompositions themselves, we must distinguish the
components obtained by complement symmetrization from the original
pairs. Theorem~\ref{thm:graphs} characterizes the former; under
the stronger $K_4$-free assumption, we can recover the latter as
well. For a maximum fractional triangle packing $p$, put
$r_e=w_e-\sum_{T\supseteq e}p_T$, and let
$C_Q(S)=\ind_{\{S\cap Q\ne\varnothing\}}$ be the coverage function
of a nonempty set $Q\subseteq V$. The decomposition associated
with $p$ is given by
\[
 f_p(S)=\sum_{e\in E}r_e C_e(S)+2\sum_{T\in\T}p_T C_T(S),
 \qquad g_p(S)=f_p(V\setminus S)-f_p(V),
\]
as in~\eqref{eq:optimizer}. Let $B$ denote the unsigned vertex-edge
incidence matrix of $G$, so that $(Bz)_i=\sum_{e\ni i}z_e$.
The next theorem describes all optimal pairs as perturbations
of these packing decompositions.

\begin{theorem}\label{thm:allopt}
Let $G=(V,E)$ be a finite simple $K_4$-free graph with nonnegative
edge weights $w$. Every optimal normalized monotonic
sum-decomposition $d_w=u+v$ has the unique form
\[
 u(S)=f_p(S)+\sum_{e\subseteq S}z_e,\qquad
 v(S)=g_p(S)-\sum_{e\subseteq S}z_e,
\]
where $p$ is a maximum fractional triangle packing and
$z\in\R^E$ satisfies $Bz=0$ and $|z_e|\le r_e$ for every $e\in E$.
Conversely, every such pair $(p,z)$ gives an optimal monotonic
sum-decomposition.
\end{theorem}

Thus the maximum packing determines the symmetrized components,
while the vector $z$ describes the remaining freedom in the
original pair. This parametrization gives the full polytope of
optimal decompositions and yields the uniqueness criterion in
Corollary~\ref{cor:uniqueness}. It also connects the classification
to the deformation cones of graphical zonotopes studied by
Padrol, Pilaud, and Poullot~\cite{PPP} and by
Poullot~\cite{Poullot}, whose work describes the rays of these
cones for $K_4$-free graphs. In our setting, complement
symmetrization imposes an additional identity that eliminates
quartic terms. This explains why the optimal-value formula
extends to $K_5$-free graphs, whereas the above description of
all optimal pairs requires the stronger $K_4$-free hypothesis.


The rest of the paper is organized as follows. {
Section~\ref{sec:prelim} collects the standard tools, develops the averaging arguments, and proves the exact finite oscillation formula.
} Section~\ref{sec:diff} treats bounded submodular differences, including the scalar criterion, its quantitative form, and the regularity and countable-block
examples. Section~\ref{sec:cuts} proves the graph formula and its
dual and computational consequences, and then classifies the
optimal decompositions on $K_4$-free graphs.

\section{Preliminaries}\label{sec:prelim}

Throughout, set algebras contain their ground sets and are closed
under complements and finite unions, and all set functions are
real-valued. For $\varphi:\A\to\R$, write
$\|\varphi\|_\infty=\sup_{A\in\A}|\varphi(A)|$; thus $\varphi$ is bounded
if and only if $\|\varphi\|_\infty<\infty$. We also recall that a \emph{monotonic
difference decomposition}, called a monotonic diff-decomposition, is a representation $\varphi=g-h$ in which $g,h$ are increasing and submodular and satisfy
$g(\varnothing)=h(\varnothing)=0$. Every such decomposition is a
bounded submodular difference, since an increasing real-valued
function on $\A$ is bounded between its values at $\varnothing$
and $J$. More generally, the chain variation of $\varphi$ is
\[
 \Var(\varphi)=
 \sup_{\varnothing=A_0\subseteq A_1\subseteq\cdots\subseteq A_r=J}
 \sum_{i=1}^r|\varphi(A_i)-\varphi(A_{i-1})|,
\]
where the supremum is over finite chains in $\A$. If $\varphi$ is
increasing, each sum telescopes, so
$\Var(\varphi)=\varphi(J)-\varphi(\varnothing)$. We write $\ind_{\mathcal E}$ for
the indicator of a condition $\mathcal E$ and
$[N]=\{1,\ldots,N\}$.

{
We first collect the standard facts used in the proofs below.
For distinct $i,j\in J$ and $S\subseteq J\setminus\{i,j\}$, write
$\Delta_{ij}\varphi(S)=\varphi(S\cup\{i,j\})-\varphi(S\cup\{i\})-\varphi(S\cup\{j\})+\varphi(S)$.

\begin{lemma}[{\cite[Propositions~2.2--2.3, Sections~6.1 and~6.3, and Appendix~B]{Bach}}]
\label{lem:submodularity}
Let $J$ be a finite set.
\begin{enumerate}
\item A function $\varphi:2^J\to\R$ is submodular if and only if
$\Delta_{ij}\varphi(S)\le0$ for all admissible $i,j,S$.
Equivalently, its marginals satisfy
$\varphi(S\cup\{i\})-\varphi(S)\ge \varphi(T\cup\{i\})-\varphi(T)$ whenever
$S\subseteq T\subseteq J\setminus\{i\}$.
\item A cardinality function $\varphi(S)=F_{|S|}$ is submodular
if and only if the successive increments of $F$ are nonincreasing.
\item Nonnegative linear combinations of submodular functions
are submodular. If $\varphi$ is submodular, then so is
$S\mapsto \varphi(J\setminus S)-\varphi(J)$; this operation reverses
monotonicity. Adding a constant preserves submodularity
and monotonicity.
\item For every nonempty $Q\subseteq J$, the coverage function
$C_Q(S)=\ind_{\{S\cap Q\ne\varnothing\}}$ is increasing,
submodular, and zero at the empty set.
\end{enumerate}
\end{lemma}

The corresponding composition rule on a set algebra will be
used for continuous profiles and for endpoint corrections.

\begin{lemma}[{\cite[Example~3.2]{Lovasz}}]\label{lem:concave-measure}
Let $m$ be a probability measure on an algebra $\A$,
and let $f:[0,1]\to\R$ be bounded and concave.
Then $\Phi_f:A\mapsto f(m(A))$ is a bounded submodular function
on $\A$. Continuity of $f$ at the endpoints is not required.
\end{lemma}

The cited concave-measure argument uses only additivity and
nonnegativity of the measure, so it applies on the stated algebra.
We also recall the scalar regularity fact needed for the
counterexamples.
}

{
For a real-valued function $f$ on $[a,b]$, write
\[
 \Var_{[a,b]}(f)
 =\sup_{a=t_0<\cdots<t_\ell=b}
   \sum_{i=1}^{\ell}|f(t_i)-f(t_{i-1})|,
\]
where the supremum is over all finite partitions of $[a,b]$.

\begin{lemma}[{\cite[p.~707]{Hartman}}]\label{lem:derivative}
If $f\in C^1([0,1])$ is a difference of two continuous
concave functions on $[0,1]$, then $f'$ has bounded variation
on every compact subinterval of $(0,1)$.
\end{lemma}
}

{
We shall also use the classical consequence that every
$C^{1,1}$ profile on $[0,1]$ is DC. To apply the
one-dimensional criterion of~\cite[p.~707]{Hartman}, extend
the profile linearly past both endpoints. The resulting $C^1$
function on $\R$ has locally Lipschitz derivative, so the
criterion applies and its DC representation restricts to $[0,1]$.
}

{
We now apply Lemma~\ref{lem:submodularity}(ii) to cardinality
averages.
} This observation applies to arbitrary submodular
functions after averaging over permutations. Indeed, for
$|J|=N$, the averaged cardinality profile
$F_k=\binom Nk^{-1}\sum_{S\subseteq J:\,|S|=k}\varphi(S)$,
$0\le k\le N$, has nonincreasing increments whenever $\varphi$ is
submodular. We use this fact first to show that a positive second
difference in the averaged profile of $\varphi=g-h$ forces a definite
oscillation in $h$.

\begin{lemma}\label{lem:finite}
Let $J$ be a finite set with $|J|=N\ge2$, and suppose that
$\varphi=g-h$ on $2^J$, where $g,h:2^J\to\R$ are submodular.
For $0\le j\le N$, let
$F_j=\binom Nj^{-1}\sum_{S\subseteq J:\,|S|=j}\varphi(S)$.
If $1\le k\le N-1$ and
$\delta=F_{k+1}-2F_k+F_{k-1}>0$, then
\begin{equation}\label{eq:finitebound}
 \osc(h)\ge\frac{k(N-k)}{N}\,\delta.
\end{equation}
In particular, if $N=2m$ and $k=m$, then
$\osc(h)\ge m\delta/2$. For every admissible pair $N,k$ and
every $\delta>0$, equality is attained by a suitable pair $g,h$.
\end{lemma}

\begin{proof}
Let $\Sym(J)$ be the permutation group of $J$. For
$g:2^J\to\R$, define its permutation average by
$\overline g(S)=\frac1{N!}\sum_{\pi\in\Sym(J)}g(\pi S)$.
If $|S|=j$, then each $j$-element subset of $J$ occurs as
$\pi S$ for exactly $j!(N-j)!$ permutations $\pi$. Consequently,
$\overline g(S)=G_j$, where
$G_j=\binom Nj^{-1}\sum_{A\subseteq J:\,|A|=j}g(A)$.
If $g$ is submodular, applying its defining inequality to
$\pi A,\pi B$ and averaging over $\pi$ gives
\[
 \overline g(A)+\overline g(B)
 \ge \overline g(A\cap B)+\overline g(A\cup B),
\]
because permutations preserve unions and intersections.
Thus $\overline g$ is also submodular. For $1\le j\le N-1$,
choose a set $R$ of size $j-1$ and distinct elements
$a,b\in J\setminus R$. Applying this inequality to
$R\cup\{a\}$ and $R\cup\{b\}$ yields
$2G_j\ge G_{j-1}+G_{j+1}$. The same argument applies to $h$.
Hence the averaged profiles
$(G_j)_{j=0}^N$ and $(H_j)_{j=0}^N$ of $g$ and $h$ are concave
sequences. Linearity of averaging also gives $F_j=G_j-H_j$
for every $j$.

Set $d_j=H_j-H_{j-1}$ for $1\le j\le N$. Concavity of $(H_j)$
means that $d_1\ge\cdots\ge d_N$, while concavity of $(G_j)$
at $k$ gives
$0\ge G_{k+1}-2G_k+G_{k-1}
=\delta+d_{k+1}-d_k$.
Therefore $d_k-d_{k+1}\ge\delta$. Since
$H_k-H_0=\sum_{j=1}^k d_j$ and
$H_N-H_k=\sum_{j=k+1}^N d_j$, monotonicity of the increments
implies $(H_k-H_0)/k\ge d_k$ and
$(H_N-H_k)/(N-k)\le d_{k+1}$. It follows that
\[
 \begin{aligned}
 H_k-\frac{N-k}{N}H_0-\frac{k}{N}H_N
 &=\frac{k(N-k)}{N}
   \left(\frac{H_k-H_0}{k}
         -\frac{H_N-H_k}{N-k}\right)\\
 &\ge\frac{k(N-k)}{N}(d_k-d_{k+1})\\
 &\ge\frac{k(N-k)}{N}\,\delta.
 \end{aligned}
\]
Let $h_{\min}=\min_{S\subseteq J}h(S)$ and
$h_{\max}=\max_{S\subseteq J}h(S)$.
Each $H_j$ is an average of values of $h$, so
$h_{\min}\le H_j\le h_{\max}$.
The same bounds hold for
$\frac{N-k}{N}H_0+\frac{k}{N}H_N$, since its coefficients
are nonnegative and sum to one. Thus the left-hand side above
is at most $h_{\max}-h_{\min}=\osc(h)$, proving
\eqref{eq:finitebound}. Substituting $N=2m$ and $k=m$
gives the stated special case.

To prove sharpness, fix $N,k$ and $\delta>0$, and define
$g\equiv0$, $\varphi=-h$, and
$h(S)=\delta(\min\{|S|,k\}-k|S|/N)$.
The cardinality profile of $h$ is
\[
 H_j=
 \begin{cases}
  \displaystyle\frac{\delta(N-k)}{N}\,j,
       &0\le j\le k,\\[4pt]
  \displaystyle\frac{\delta k}{N}\,(N-j),
       &k\le j\le N.
 \end{cases}
\]
Its successive increments are $\delta(N-k)/N$ for
$1\le j\le k$ and $-\delta k/N$ for $k+1\le j\le N$.
They are nonincreasing, so $h$ is submodular; the constant
function $g$ is also submodular. The drop between the $k$th
and $(k+1)$st increments is $\delta$, and therefore
$F_{k+1}-2F_k+F_{k-1}
=-(H_{k+1}-2H_k+H_{k-1})=\delta$.
Finally, the displayed profile is nonnegative, vanishes at
$j=0,N$, and attains its maximum
$\delta k(N-k)/N$ at $j=k$.
Hence $\osc(h)=\delta k(N-k)/N$, establishing equality
in~\eqref{eq:finitebound} for every admissible $N,k$ and
every $\delta>0$.
\end{proof}

Lemma~\ref{lem:finite} applies to arbitrary submodular components
$g,h$; no symmetry is assumed, and averages taken on different
finite Boolean algebras need not be compatible. Its bound uses
only one positive second difference of the averaged profile.
We now combine all positive second differences to obtain the
exact minimum oscillation when the function being decomposed
depends only on cardinality. For arbitrary functions, the same
argument yields a lower bound in terms of the averaged profile.

Let $J$ be a finite set of size $N\ge2$, and let
$\varphi(S)=F_{|S|}$, where $F_0,\ldots,F_N$ are arbitrary real
numbers. For $1\le k\le N-1$ and $0\le j\le N$, put
\[
 \delta_k=(F_{k+1}-2F_k+F_{k-1})_+,\qquad
 K_N(j,k)=\min\{j,k\}-\frac{jk}{N},
\]
where $x_+=\max\{x,0\}$. The kernel $K_N$ is nonnegative and
vanishes at $j=0,N$. The next theorem expresses the optimal
oscillation as the maximum of the sum of these kernels weighted
by the positive second differences.

\begin{theorem}\label{thm:finite}
Let $J$ be a finite set with $|J|=N\ge2$, and put
$K_N(j,k)=\min\{j,k\}-jk/N$ for $0\le j\le N$ and
$1\le k\le N-1$. For the cardinality function
$\varphi(S)=F_{|S|}$, define
$\delta_k=(F_{k+1}-2F_k+F_{k-1})_+$. Then
\[
 \min_{\substack{\varphi=g-h\\ g,h\ \text{submodular}}}\osc(h)
 =\max_{0\le j\le N}
   \sum_{k=1}^{N-1}K_N(j,k)\delta_k,
\]
and the minimum is attained by cardinality functions $g,h$.
For an arbitrary function $\varphi:2^J\to\R$, let
$F_j=\binom Nj^{-1}\sum_{S\subseteq J:\,|S|=j}\varphi(S)$
and define $\delta_k$ from this averaged profile.
Every submodular difference decomposition $\varphi=g-h$ then satisfies
\[
 \osc(h)\ge
 \max_{0\le j\le N}
 \sum_{k=1}^{N-1}K_N(j,k)\delta_k.
\]
\end{theorem}

\begin{proof}
We first prove the lower bound for an arbitrary function $\varphi$.
Suppose that $\varphi=g-h$, with $g,h$ submodular, and let
$G_j,H_j$ be their averages over the $j$-element subsets of $J$.
Averaging over permutations preserves submodularity and replaces
each function by its cardinality average, as shown in the proof
of Lemma~\ref{lem:finite}. Consequently, $(G_j)_{j=0}^N$ and
$(H_j)_{j=0}^N$ are concave sequences, and linearity of averaging
gives $F_j=G_j-H_j$. Define
$\osc(H)=\max_{0\le j\le N}H_j-\min_{0\le j\le N}H_j$.
Each $H_j$ is an average of values of $h$, so
$\min_{S\subseteq J}h(S)\le H_j\le\max_{S\subseteq J}h(S)$
and hence $\osc(H)\le\osc(h)$.
For $1\le k\le N-1$, put
$\kappa_k=2H_k-H_{k-1}-H_{k+1}$.
Concavity of $H$ gives $\kappa_k\ge0$, while concavity of
$G=F+H$ gives
\[
 0\ge G_{k+1}-2G_k+G_{k-1}
   =F_{k+1}-2F_k+F_{k-1}-\kappa_k.
\]
Thus $\kappa_k\ge\delta_k$ for every $k$.

To express $H$ in terms of these quantities, observe that
\[
 K_N(j,k)=
 \begin{cases}
  \displaystyle\frac{j(N-k)}{N},&0\le j\le k,\\[4pt]
  \displaystyle\frac{k(N-j)}{N},&k\le j\le N.
 \end{cases}
\]
In particular, $K_N(j,k)\ge0$ and
$K_N(0,k)=K_N(N,k)=0$.
For fixed $k$, its successive increments in $j$ are
$(N-k)/N$ up to index $k$ and $-k/N$ thereafter.
Their only drop is therefore a drop of one at $k$, which gives
\[
 2K_N(j,k)-K_N(j-1,k)-K_N(j+1,k)
 =\ind_{\{j=k\}}
 \qquad(1\le j\le N-1).
\]
Let $L_j=(1-j/N)H_0+(j/N)H_N$ and
$Q_j=\sum_{k=1}^{N-1}K_N(j,k)\kappa_k$.
Both $H-L$ and $Q$ vanish at $j=0,N$, and the preceding
identity shows that both have negative second difference
$\kappa_j$ at each interior index $j$.
Their difference $D_j=H_j-L_j-Q_j$ consequently satisfies
$D_{j+1}-D_j=D_j-D_{j-1}$ for $1\le j\le N-1$.
Its increments are constant, and $D_0=D_N=0$ forces this
constant to be zero. Hence
\[
 H_j-L_j=\sum_{k=1}^{N-1}K_N(j,k)\kappa_k
 \qquad(0\le j\le N).
\]
Since $L_j$ is a convex combination of $H_0,H_N$, it lies
between the minimum and maximum entries of $H$. Therefore
$H_j-L_j\le\osc(H)$, and nonnegativity of the kernel gives
\[
 \osc(h)\ge\osc(H)\ge H_j-L_j
 =\sum_{k=1}^{N-1}K_N(j,k)\kappa_k
 \ge\sum_{k=1}^{N-1}K_N(j,k)\delta_k.
\]
Taking the maximum over $j$ proves the lower bound, including
the assertion for arbitrary $\varphi$.

Now suppose that $\varphi(S)=F_{|S|}$. Define
\[
 H_j^*=\sum_{k=1}^{N-1}K_N(j,k)\delta_k,
 \qquad G_j^*=F_j+H_j^*
 \qquad(0\le j\le N).
\]
The kernel identity implies
$H_{j+1}^*-2H_j^*+H_{j-1}^*=-\delta_j\le0$
for $1\le j\le N-1$. Moreover,
\[
 G_{j+1}^*-2G_j^*+G_{j-1}^*
 =F_{j+1}-2F_j+F_{j-1}-\delta_j
 =\min\{F_{j+1}-2F_j+F_{j-1},0\}\le0.
\]
Thus both profiles are concave, so the cardinality functions
$h(S)=H_{|S|}^*$ and $g(S)=G_{|S|}^*$ are submodular.
By construction, $g(S)-h(S)=F_{|S|}=\varphi(S)$ for every
$S\subseteq J$. Finally, nonnegativity and the endpoint values
of the kernel give $H_j^*\ge0$ and $H_0^*=H_N^*=0$.
Since every cardinality $0,\ldots,N$ occurs among subsets of $J$,
\[
 \osc(h)=\max_{0\le j\le N}H_j^*
 =\max_{0\le j\le N}
   \sum_{k=1}^{N-1}K_N(j,k)\delta_k.
\]
This decomposition attains the lower bound and proves the
claimed minimum formula.
\end{proof}

Theorem~\ref{thm:finite} recovers Lemma~\ref{lem:finite} by
retaining only the $k$th summand and evaluating at $j=k$, since
$K_N(k,k)=k(N-k)/N$. Combining several positive second differences
can yield a strictly stronger bound. For example, if $N=4$ and
$(F_0,\ldots,F_4)=(0,0,1,2,4)$, then
$(\delta_1,\delta_2,\delta_3)=(1,0,1)$.
The best bound obtained from Lemma~\ref{lem:finite} is $3/4$,
whereas $K_4(j,1)+K_4(j,3)=1$ for $j=1,2,3$, so
Theorem~\ref{thm:finite} gives the exact cost $1$.

We next develop the averaging argument for graph cuts.
Here the relevant symmetry is complementation, which allows us
to represent a monotonic sum-decomposition by a single increasing
submodular function without changing its objective value.
Before establishing this reduction, we verify that the minimum
defining $N_G(w)$ is attained. Retaining the notation of the
introduction, each weighted edge cut has the monotonic
sum-decomposition
$w_e\ind_{\{S\cap e\ne\varnothing\}}-w_e\ind_{\{e\subseteq S\}}$.
Summing over edges therefore gives a feasible pair with
$u(V)=W$. The feasible set is closed, since submodularity,
monotonicity, normalization, and the identity $u+v=d_w$ impose
finitely many linear equalities and inequalities on the values
of $u,v$. On a sublevel set $u(V)\le M$, monotonicity and
normalization give $0\le u(S)\le M$ for every $S\subseteq V$.
Since $d_w(V)=0$, we also have $v(V)=-u(V)$ and hence
$-M\le v(S)\le0$. Thus every nonempty sublevel set is compact.
In particular, the sublevel set $u(V)\le W$ is nonempty and
compact, so the objective attains its global minimum there.
The following lemma shows that complement averaging preserves
this minimum.

\begin{lemma}\label{lem:cutsym}
Let $G=(V,E)$ be a finite graph with nonnegative edge weights
$w$, and let $d_w=u+v$ be a monotonic sum-decomposition with
$u(\varnothing)=v(\varnothing)=0$.
Then there exists an increasing submodular function
$f:2^V\to\R$ satisfying $f(\varnothing)=0$, $f(V)=u(V)$, and
\begin{equation}\label{eq:cutidentity}
 d_w(S)=f(S)+f(V\setminus S)-f(V)
 \qquad(S\subseteq V).
\end{equation}
Conversely, if an increasing submodular function $f:2^V\to\R$
satisfies $f(\varnothing)=0$ and~\eqref{eq:cutidentity}, then
$d_w=f+v_f$, where $v_f(S)=f(V\setminus S)-f(V)$, is a
monotonic sum-decomposition with both components zero at the
empty set and objective value $f(V)$.
\end{lemma}

\begin{proof}
{
For $\varphi:2^V\to\R$, write $\varphi^c(S)=\varphi(V\setminus S)$.
We use the closure properties in Lemma~\ref{lem:submodularity}(iii).
}

Now suppose that $d_w=u+v$ is a normalized monotonic
sum-decomposition. Since $d_w(V)=0$, we have
$u(V)+v(V)=0$. The function
$\widetilde u(S)=v(V\setminus S)-v(V)$ is increasing and
submodular, because $v$ is decreasing and submodular.
Moreover, $\widetilde u(\varnothing)=0$ and
$\widetilde u(V)=v(\varnothing)-v(V)=u(V)$.
Define
\[
 f(S)=\frac{u(S)+\widetilde u(S)}2
     =\frac{u(S)+v(V\setminus S)-v(V)}2.
\]
As an average of two increasing submodular functions, $f$
is increasing and submodular. The endpoint values of $u$
and $\widetilde u$ also give $f(\varnothing)=0$ and
$f(V)=u(V)$. To verify the required identity, observe that
$S$ and $V\setminus S$ determine the same cut, so
$d_w(S)=d_w(V\setminus S)$. Consequently,
\[
 \begin{aligned}
 f(S)+f(V\setminus S)-f(V)
 &=\frac{u(S)+v(S)+u(V\setminus S)+v(V\setminus S)}2
      -v(V)-u(V)\\
 &=\frac{d_w(S)+d_w(V\setminus S)}2\\
 &=d_w(S).
 \end{aligned}
\]
This proves~\eqref{eq:cutidentity} and shows that the
objective value is preserved.

Conversely, let $f$ satisfy the stated assumptions and put
$v_f(S)=f(V\setminus S)-f(V)$.
The complementation argument above shows that $v_f$ is
submodular and decreasing. Also,
$v_f(\varnothing)=f(V)-f(V)=0$, while
$f(\varnothing)=0$ by assumption.
Equation~\eqref{eq:cutidentity} gives
$f(S)+v_f(S)=d_w(S)$ for every $S\subseteq V$.
Thus $d_w=f+v_f$ is a normalized monotonic sum-decomposition,
and its increasing component has value $f(V)$ at $V$.
\end{proof}

{
For later use, we record the following results on graph cuts
in the notation of this paper. The clique support statement is
\cite[Lemma~7.2]{BGILS}; the construction and upper bound are
from~\cite[Theorem~7.7 and Corollary~7.8]{BGILS}; and the
complete-graph value follows from~\cite[Lemma~7.6]{BGILS}.

\begin{lemma}\label{lem:known-cuts}
Let $G=(V,E)$ be a finite simple graph with nonnegative
edge weights $w$, and put $W=\sum_{e\in E}w_e$.
\begin{enumerate}
\item If $d_w=u+v$ is a normalized monotonic sum-decomposition,
then the unique multilinear representation
\[
 u(x)=\sum_{U\subseteq V}a_U\prod_{i\in U}x_i,
 \qquad
 a_U=\sum_{R\subseteq U}(-1)^{|U|-|R|}u(R)
\]
has $a_U=0$ whenever $U$ is not a clique of $G$.
\item For every feasible fractional triangle packing $p$, put
$r_e=w_e-\sum_{T\supseteq e}p_T$. The functions
\begin{equation}\label{eq:optimizer}
 \begin{aligned}
 f_p(S)&=\sum_{e\in E}r_eC_e(S)
          +2\sum_{T\in\T}p_TC_T(S),\\
 g_p(S)&=f_p(V\setminus S)-f_p(V).
 \end{aligned}
\end{equation}
form a normalized monotonic sum-decomposition of $d_w$, with
$f_p(V)=W-\sum_Tp_T$. Consequently,
$N_G(w)\le W-\nu_w^*(G)$.
\item For the unweighted complete graph $K_n$,
$N_{K_n}=\max_{S\subseteq V(K_n)}d_{K_n}(S)=\lfloor n^2/4\rfloor$.
\end{enumerate}
\end{lemma}

The final standard fact concerns the dimension of the
perturbation space in the classification of optimal decompositions.

\begin{lemma}[{\cite[Theorem~8.2.1]{GodsilRoyle}}]
\label{lem:incidence-rank}
Let $G=(V,E)$ be a finite simple graph, let $B$ be its unsigned
vertex-edge incidence matrix, and let $b(G)$ be the number
of its bipartite connected components, including isolated vertices.
Then $\operatorname{rank}_{\R}B=|V|-b(G)$ and
$\dim_{\R}\ker B=|E|-|V|+b(G)$.
\end{lemma}
}

\section{Bounded submodular differences}\label{sec:diff}

We begin with the scalar characterization in
Theorem~\ref{thm:main}, allowing arbitrary bounded submodular
components. Finite averaging reduces the necessity argument to
concave profiles, without requiring compatibility between the
underlying partitions. We then establish the exact oscillation
formula in Theorem~\ref{thm:continuous} and remove the continuity
assumption in Proposition~\ref{prop:automatic}. These results
provide the basis for the counterexamples: an oscillatory $C^1$
profile first exhibits the obstruction, and subsequent
constructions strengthen its regularity. Finally, the finite
oscillation bounds lead from H\"older and majority examples to
the exact criterion for countable block functions.

\begin{proof}[Proof of the scalar criterion in Theorem~\ref{thm:main}]
Suppose that $\Phi_f=g-h$, where $g,h$ are bounded submodular
functions on $\B$. Choose $M>0$ such that
$|g(A)|,|h(A)|\le M$ for every $A\in\B$.
For each integer $N\ge2$, choose a Borel partition
$P_{N,1},\ldots,P_{N,N}$ of $[0,1]$ with
$\lambda(P_{N,i})=1/N$. The map
$S\mapsto\bigcup_{i\in S}P_{N,i}$ from $2^{[N]}$ to $\B$
preserves unions and intersections. Composing $g,h$ with this
map therefore gives submodular functions on $2^{[N]}$.
Let $G_{N,k},H_{N,k}$ be their averages over the $k$-element
subsets of $[N]$. By the averaging argument in
Lemma~\ref{lem:finite}, both profiles are concave sequences.
Moreover, $|G_{N,k}|,|H_{N,k}|\le M$, and every union of $k$
cells has measure $k/N$, so
$G_{N,k}-H_{N,k}=f(k/N)$.
Let $G_N,H_N$ be the piecewise linear interpolants of these
profiles at the points $k/N$, $0\le k\le N$.
Their slopes are nonincreasing, hence both functions are
concave on $[0,1]$, and interpolation preserves the bounds
$|G_N|,|H_N|\le M$. Their difference is the piecewise linear
interpolant $f_N$ of $f$ on the same grid. If
$\omega_f(\eta)=\sup\{|f(s)-f(t)|:s,t\in[0,1],\
|s-t|\le\eta\}$, then
$\|f_N-f\|_\infty\le\omega_f(1/N)\to0$.
Indeed, on each grid interval, $f_N$ is a convex combination
of the endpoint values of $f$, and each endpoint is within
distance $1/N$ of every point in that interval.

To obtain limiting concave components, fix
$0<\varepsilon<1/2$. If $q$ is concave on $[0,1]$ and
$|q|\le M$, then its secant slopes satisfy
\[
 \frac{q(1)-q(y)}{1-y}
 \le \frac{q(y)-q(x)}{y-x}
 \le \frac{q(x)-q(0)}{x}
 \qquad(\varepsilon\le x<y\le1-\varepsilon).
\]
The outer quotients bound the middle one between
$-2M/\varepsilon$ and $2M/\varepsilon$.
Thus $G_N,H_N$ are uniformly bounded and uniformly Lipschitz
on every compact subinterval of $(0,1)$.
Apply the Arzel\`a--Ascoli theorem simultaneously to the pairs
$(G_N,H_N)$ on $[1/m,1-1/m]$, successively for $m\ge3$.
A diagonal subsequence then gives integers $N_\ell\to\infty$
such that $G_{N_\ell}\to G$ and $H_{N_\ell}\to H$ locally
uniformly on $(0,1)$. The limits satisfy $|G|,|H|\le M$,
and passing to the limit in the concavity inequalities shows
that both are concave. Since
$G_{N_\ell}-H_{N_\ell}=f_{N_\ell}$ and
$f_{N_\ell}\to f$ uniformly, we obtain $G-H=f$ on $(0,1)$.
Only the common bounds and concavity of the interpolants
are used here; the partitions chosen for different $N$
need not refine one another.

We next extend $G,H$ to the endpoints. Let $q$ denote either
limit function, and fix $b\in(0,1)$. For $0<x<b$, put
$\sigma_b(x)=(q(b)-q(x))/(b-x)$.
If $0<x<y<b$, concavity gives
\[
 q(y)\ge
 \frac{b-y}{b-x}q(x)+\frac{y-x}{b-x}q(b),
\]
and rearranging yields $\sigma_b(y)\le\sigma_b(x)$.
Thus $\sigma_b$ is nonincreasing. For $0<x<b/2$,
the bound $|q|\le M$ gives $|\sigma_b(x)|\le4M/b$,
so $\sigma_b(x)$ has a finite limit as $x\downarrow0$.
The identity $q(x)=q(b)-(b-x)\sigma_b(x)$ then shows that
$q(x)$ also has a finite limit at $0$.
Applying the same argument to the bounded concave function
$x\mapsto q(1-x)$ gives a finite limit at $1$.
Define the endpoint values of $G,H$ by these one-sided limits.
The resulting extensions are continuous on $[0,1]$ and remain
concave: approximate any endpoint arguments by interior points
and pass to the limit in the concavity inequality.
Since $f$ is continuous and $G-H=f$ in the interior,
the same equality holds at both endpoints.
This proves necessity. The extended endpoint values need not
coincide with the endpoint values of the interpolants
$G_N,H_N$; only their difference is prescribed by $f$.

{
Conversely, suppose that $f=G-H$, where $G,H$ are continuous
and concave on $[0,1]$. Lemma~\ref{lem:concave-measure}
shows that $G\circ\lambda$ and $H\circ\lambda$ are bounded
submodular functions with difference $\Phi_f$, proving sufficiency.
}
\end{proof}
The finite averaging argument also yields the quantitative
characterization in Theorem~\ref{thm:continuous}. Its key additional
feature is that it preserves bounds on the range of the subtracted
component. We combine this observation with a kernel representation
of continuous concave functions to determine the optimal oscillation.

\begin{proof}[Proof of Theorem~\ref{thm:continuous}]
We first extend the scalar characterization to the given atomless
probability space $(X,\A,m)$. The divisibility property of
an atomless finite measure provides, for every $N\ge2$, a measurable
partition of $X$ into $N$ cells of measure $1/N$.
The necessity argument in Theorem~\ref{thm:main} therefore applies:
restrict a bounded submodular decomposition $\Phi_f=g-h$ to the
algebra generated by each partition, average over permutations,
and take locally uniform limits of the resulting concave
interpolants. Extending the limits continuously to the endpoints
gives $f=G-H$ with $G,H$ continuous and concave on $[0,1]$.
This construction also preserves the required oscillation bound.
Indeed, if $a=\inf_{A\in\A}h(A)$ and
$b=\sup_{A\in\A}h(A)$, then every averaged value,
every interpolated value, and every value of the limiting
extension $H$ lies in $[a,b]$. Hence
$\osc(H)\le b-a=\osc(h)$, where
$\osc(H)=\max_{[0,1]}H-\min_{[0,1]}H$.
Conversely, composing continuous concave components with $m$
{
gives bounded submodular functions by Lemma~\ref{lem:concave-measure}.
} This proves the equivalence of
(i) and (ii).

We next establish the kernel representation needed for both
the curvature criterion and the optimal value.
We use two elementary distributional facts: a continuous function
on an interval is concave if and only if its second distributional
derivative is nonpositive, and a continuous function with zero
second distributional derivative is affine.
For the first assertion, convolution with nonnegative smooth
mollifiers on compact subintervals converts the distributional
inequality into the usual second-derivative inequality; local
uniform convergence then preserves concavity. The converse
follows by the same approximation, since mollification preserves
concavity. Applying this characterization to both a function
and its negative proves the second assertion.
For fixed $t\in(0,1)$, the kernel
$K(x,t)=\min\{x,t\}-xt$ has the form
\[
 K(x,t)=
 \begin{cases}
  x(1-t),&0\le x\le t,\\
  t(1-x),&t\le x\le1.
 \end{cases}
\]
It is continuous and concave in $x$, vanishes at $x=0,1$,
and satisfies $0\le K(x,t)\le t(1-t)$.
Its first derivative drops by one at $x=t$, so integration
by parts on $(0,t)$ and $(t,1)$ gives
\[
 \int_0^1 K(x,t)\phi''(x)\,dx=-\phi(t)
 \qquad\bigl(\phi\in C_c^\infty(0,1)\bigr).
\]

Let $q$ be continuous and concave on $[0,1]$.
Its curvature $\nu=-q''$ is a positive locally finite Radon
measure on $(0,1)$; equivalently, it is the measure associated
with the nonincreasing right derivative of $q$.
Write $L_q(x)=(1-x)q(0)+xq(1)$.
For a compact interval $I\subset(0,1)$, define
$Q_I(x)=\int_I K(x,t)\,d\nu(t)$.
Since $\nu(I)<\infty$, this function is continuous on $[0,1]$,
vanishes at the endpoints, and satisfies
$Q_I''=-\nu|_I$ by the kernel identity and Fubini's theorem.
Thus $q-L_q-Q_I$ has nonpositive second distributional
derivative and is concave. Its endpoint values are zero,
so it is nonnegative throughout $[0,1]$.
Using $K(1/2,t)\ge\frac12t(1-t)$, we obtain
\[
 \int_I t(1-t)\,d\nu(t)
 \le 2Q_I(1/2)
 \le 2\left(q(1/2)-\frac{q(0)+q(1)}2\right).
\]
Exhausting $(0,1)$ by increasing compact intervals and applying
monotone convergence proves
$\int_{(0,1)}t(1-t)\,d\nu(t)<\infty$.

Conversely, let $\nu$ be a positive locally finite Radon
measure on $(0,1)$ with
$\int_{(0,1)}t(1-t)\,d\nu(t)<\infty$, and put
$Q_\nu(x)=\int_{(0,1)}K(x,t)\,d\nu(t)$.
The bound $0\le K(x,t)\le t(1-t)$ and dominated convergence
show that $Q_\nu$ is finite and continuous on $[0,1]$, with
$Q_\nu(0)=Q_\nu(1)=0$. Integrating the concavity inequality
for $K(\cdot,t)$ shows that $Q_\nu$ is concave.
Moreover, for $\phi\in C_c^\infty(0,1)$,
\[
 \int_{(0,1)}\int_0^1
 K(x,t)|\phi''(x)|\,dx\,d\nu(t)
 \le \|\phi''\|_{L^1(0,1)}
       \int_{(0,1)}t(1-t)\,d\nu(t)<\infty.
\]
Fubini's theorem is therefore applicable even when
$\nu((0,1))=\infty$, and it gives
\[
 \langle Q_\nu'',\phi\rangle
 =\int_{(0,1)}
   \left(\int_0^1K(x,t)\phi''(x)\,dx\right)d\nu(t)
 =-\int_{(0,1)}\phi(t)\,d\nu(t).
\]
Hence $Q_\nu''=-\nu$.
Applying this to the curvature measure of the preceding
concave function $q$, we find that $q-L_q-Q_\nu$ has zero
second distributional derivative. It is therefore affine,
and its zero endpoint values imply the representation
\[
 q(x)-L_q(x)=\int_{(0,1)}K(x,t)\,d\nu(t).
\]

We can now prove the equivalence of (ii) and (iii).
If $f=G-H$ with continuous concave components, set
$\nu_G=-G''$ and $\nu_H=-H''$.
Then $\mu=f''=\nu_H-\nu_G$ is a signed locally finite
Radon measure, and $|\mu|\le\nu_H+\nu_G$.
The weighted integrability just proved for each curvature
measure gives
$\int_{(0,1)}t(1-t)\,d|\mu|(t)<\infty$.
Conversely, suppose that (iii) holds, and let $\mu^+,\mu^-$
be the positive and negative parts of $\mu$, defined by its
compatible Jordan decompositions on compact subintervals.
Both satisfy the weighted integrability condition. Define
\[
 H_+(x)=\int_{(0,1)}K(x,t)\,d\mu^+(t),
 \qquad
 H_-(x)=\int_{(0,1)}K(x,t)\,d\mu^-(t).
\]
These functions are continuous and concave and vanish at
both endpoints. For $L_f(x)=(1-x)f(0)+xf(1)$, the function
$L_f+H_--H_+$ has second distributional derivative
$-\mu^-+\mu^+=\mu=f''$ and the same endpoint values as $f$.
Their difference is therefore affine and zero at both
endpoints, proving
\[
 f=L_f+H_--H_+.
\]
Since $L_f+H_-$ is concave, this establishes (ii).
It also gives a bounded submodular decomposition
$\Phi_f=(L_f+H_-)\circ m-H_+\circ m$.
The function $H_+$ is nonnegative and vanishes at the endpoints,
so $\osc(H_+)=\max_{x\in[0,1]}H_+(x)$.
An atomless probability measure takes every value in $[0,1]$,
and hence $\osc(H_+\circ m)=\osc(H_+)$.
Consequently, $\mathcal C(f)\le\max_{x\in[0,1]}H_+(x)$.

For the reverse inequality, take any bounded submodular
decomposition $\Phi_f=g-h$. The averaging argument at the
start of the proof gives continuous concave functions $G,H$
such that $f=G-H$ and $\osc(H)\le\osc(h)$.
Put $\nu_G=-G''$ and $\nu_H=-H''$.
These are positive locally finite measures, and
$\mu=\nu_H-\nu_G$. We claim that $\mu^+\le\nu_H$.
Indeed, for every relatively compact Borel set
$E\subset(0,1)$, the positive-part formula for the finite
signed measure on $E$ gives
\[
 \mu^+(E)
 =\sup_{\substack{B\subseteq E\\ B\ \mathrm{Borel}}}\mu(B)
 \le\sup_{\substack{B\subseteq E\\ B\ \mathrm{Borel}}}\nu_H(B)
 \le\nu_H(E).
\]
Exhaustion by compact subintervals extends this domination
to all Borel sets.
The kernel representation of $H$ and nonnegativity of $K$
therefore yield
\[
 H(x)-(1-x)H(0)-xH(1)
 =\int_{(0,1)}K(x,t)\,d\nu_H(t)
 \ge H_+(x).
\]
The endpoint interpolation $(1-x)H(0)+xH(1)$ lies between
the minimum and maximum of $H$, so the left-hand side is
at most $\osc(H)$. Thus
$\osc(h)\ge\osc(H)\ge H_+(x)$ for every $x\in[0,1]$.
Taking the maximum over $x$ and then the infimum over all
decompositions proves
\[
 \mathcal C(f)
 =\max_{x\in[0,1]}
   \int_{(0,1)}K(x,t)\,d\mu^+(t).
\]
The decomposition constructed above attains this value,
and both components depend only on measure.
\end{proof}

\begin{remark}
The weighted curvature condition allows infinite curvature
mass near the endpoints. For example, $f(t)=\sqrt t$ is
continuous and concave on $[0,1]$, while
$-f''(t)=\frac14t^{-3/2}$ on $(0,1)$.
Thus $\int_{(0,1)}d|f''|=\infty$, but
$\int_{(0,1)}t(1-t)\,d|f''|(t)<\infty$.
Requiring finite unweighted curvature, or finite total
variation of the derivative over the whole open interval,
would therefore exclude admissible profiles.
\end{remark}

We next remove the continuity assumption. The main step is to
show that bounded submodular components force the scalar profile
to be locally Lipschitz in the interior. The compactness argument
then gives a continuous DC extension, while separate endpoint
corrections account for the remaining freedom.

\begin{proof}[Proof of Proposition~\ref{prop:automatic}]
Suppose that $f\circ m=g-h$, where $g,h$ are bounded submodular
functions on $\A$, and put
$C=\osc(g)+\osc(h)$.
Since an atomless probability measure takes every value in
$[0,1]$, the range of $f$ equals the range of $f\circ m$.
Thus $f$ is bounded and $\osc(f)\le C$.
Fix $0<\varepsilon<1/2$ and take
$\varepsilon\le x<y\le1-\varepsilon$, with $d=y-x$.
We first consider $0<d\le\varepsilon/2$ and set
\[
 k=\lfloor x/d\rfloor,\qquad
 a=x-kd,\qquad
 N=\lfloor(1-a)/d\rfloor.
\]
Then $0\le a<d$, $x=a+kd$, $y=a+(k+1)d$, and $a+Nd\le1$.
Since $y\le1$, we also have $k+1\le N$.
By atomlessness, there are pairwise disjoint measurable sets
$B_0,B_1,\ldots,B_N$ with $m(B_0)=a$ and $m(B_i)=d$ for
$1\le i\le N$; when $a=0$, take $B_0=\varnothing$.
For $S\subseteq[N]$, put
$A_S=B_0\cup\bigcup_{i\in S}B_i$.
The identities $A_S\cap A_T=A_{S\cap T}$ and
$A_S\cup A_T=A_{S\cup T}$ show that
$S\mapsto g(A_S)$ and $S\mapsto h(A_S)$ are submodular.
Averaging these functions over permutations of $[N]$ gives
concave profiles $(G_j)_{j=0}^N$ and
$(H_j)_{j=0}^N$ satisfying
$G_j-H_j=f(a+jd)$.
Moreover, averaging does not increase oscillation, so
$\osc(G)\le\osc(g)$ and $\osc(H)\le\osc(h)$,
where the oscillation of a finite sequence is its maximum
minus its minimum.

For either profile $Q=(Q_j)_{j=0}^N$, let
$D_j=Q_j-Q_{j-1}$.
Concavity means that $D_1\ge\cdots\ge D_N$, and hence
\[
 (k+1)D_{k+1}
 \le \sum_{j=1}^{k+1}D_j
 =Q_{k+1}-Q_0
 \le\osc(Q),
\]
while
\[
 (N-k)D_{k+1}
 \ge \sum_{j=k+1}^{N}D_j
 =Q_N-Q_k
 \ge-\osc(Q).
\]
Consequently,
$-\osc(Q)/(N-k)\le Q_{k+1}-Q_k\le\osc(Q)/(k+1)$.
The definitions of $k,N$ give
$(k+1)d>x\ge\varepsilon$ and
\[
 (N-k)d\ge1-a-d-kd
 =1-x-d\ge\varepsilon-d\ge\varepsilon/2.
\]
It follows that
$|Q_{k+1}-Q_k|\le2d\,\osc(Q)/\varepsilon$.
Applying this to $G,H$ yields
\[
 \begin{aligned}
 |f(y)-f(x)|
 &\le|G_{k+1}-G_k|
      +|H_{k+1}-H_k|\\
 &\le\frac{2d}{\varepsilon}
       \bigl(\osc(G)+\osc(H)\bigr)
 \le\frac{2C}{\varepsilon}|y-x|.
 \end{aligned}
\]
If $d>\varepsilon/2$, the same bound follows from
$|f(y)-f(x)|\le\osc(f)\le C$.
Thus $f$ is Lipschitz on $[\varepsilon,1-\varepsilon]$
and therefore locally Lipschitz on $(0,1)$.

We now return to the equal-measure partitions used in the
proof of Theorem~\ref{thm:main}.
Choose $M>0$ with $|g(A)|,|h(A)|\le M$ for all
$A\in\A$.
For each $N\ge2$, partition $X$ into $N$ measurable cells
of measure $1/N$, average the restrictions of $g,h$, and
let $G_N,H_N$ be the corresponding piecewise linear
interpolants on the grid $\{j/N:0\le j\le N\}$.
These functions are concave, satisfy $|G_N|,|H_N|\le M$,
and have difference equal to the interpolant $f_N$ of $f$.
Although $f$ need not be continuous at the endpoints,
$f_N\to f$ locally uniformly on $(0,1)$.
Indeed, if $N>2/\varepsilon$, every grid interval meeting
$[\varepsilon,1-\varepsilon]$ has its endpoints in
$[\varepsilon/2,1-\varepsilon/2]$. If $L_{\varepsilon/2}$
is a Lipschitz constant for $f$ on this larger interval,
then
\[
 \sup_{x\in[\varepsilon,1-\varepsilon]}|f_N(x)-f(x)|
 \le\frac{L_{\varepsilon/2}}{N}\longrightarrow0.
\]
The secant-slope bounds from the proof of
Theorem~\ref{thm:main} make $G_N,H_N$ uniformly Lipschitz
on each compact subinterval of $(0,1)$.
Applying Arzel\`a--Ascoli simultaneously to both sequences
and taking a diagonal subsequence gives bounded concave
limits $G,H$ with $G-H=f$ on $(0,1)$.
As proved there, bounded concave functions on $(0,1)$ have
finite one-sided endpoint limits, and extension by these
limits preserves continuity and concavity.
Thus $\widetilde f=G-H$ is a continuous DC function on
$[0,1]$ agreeing with $f$ throughout $(0,1)$.
No agreement with the original values $f(0),f(1)$ is
required.

Conversely, suppose that $f|_{(0,1)}$ has a continuous
DC extension $\widetilde f=G-H$, with $G,H$ continuous
and concave on $[0,1]$.
For $i=0,1$, put $c_i=f(i)-\widetilde f(i)$ and
$q_i(t)=-\ind_{\{t=i\}}$.
Each $q_i$ is bounded and concave. To verify concavity,
take $s,t\in[0,1]$ and $0<\theta<1$.
If $(1-\theta)s+\theta t\ne i$, then
$q_i((1-\theta)s+\theta t)=0$, whereas
$(1-\theta)q_i(s)+\theta q_i(t)\le0$.
If $(1-\theta)s+\theta t=i$, the fact that $i$ is an
endpoint forces $s=t=i$, so equality holds.
The cases $\theta=0,1$ are immediate.
Since $f-\widetilde f$ vanishes in the interior and has
value $c_i$ at endpoint $i$, we have
\[
 f-\widetilde f
 =\sum_{i=0}^1(-c_i)_+q_i
  -\sum_{i=0}^1(c_i)_+q_i.
\]
Indeed, $(-c_i)_+-(c_i)_+=-c_i$ and
$q_i=-\ind_{\{t=i\}}$.
Therefore the functions
\[
 \widehat G=G+\sum_{i=0}^1(-c_i)_+q_i,
 \qquad
 \widehat H=H+\sum_{i=0}^1(c_i)_+q_i
\]
are bounded and concave on $[0,1]$ and satisfy
$f=\widehat G-\widehat H$.
{
They may be discontinuous at the endpoints, but
Lemma~\ref{lem:concave-measure} requires only boundedness
and concavity of the profile.
}
Hence $\widehat G\circ m$ and $\widehat H\circ m$ are
bounded submodular functions whose difference is $f\circ m$.
This proves sufficiency and shows that both endpoint values
are unrestricted.
\end{proof}

The preceding argument shows that bounded submodular
decomposability forces a profile to be locally Lipschitz.
For $C^1$ profiles, the DC characterization imposes a stronger
restriction: the derivative must have bounded variation on
every compact subinterval of $(0,1)$.
{
This is the property recorded in Lemma~\ref{lem:derivative}.
}

We now use Lemma~\ref{lem:derivative} to construct the counterexample
in Theorem~\ref{thm:main}. The profile will have a positive bounded
continuous derivative, but the derivative will have infinite
variation on an interior compact interval.

\begin{proof}[Proof of the counterexample in Theorem~\ref{thm:main}]
Define
\[
 q(t)=
 \begin{cases}
  1+(t-\tfrac12)\sin\bigl((t-\tfrac12)^{-2}\bigr),
       &t\ne\tfrac12,\\
  1,&t=\tfrac12,
 \end{cases}
 \qquad
 F(t)=\int_0^t q(s)\,ds.
\]
The function $q$ is continuous away from $1/2$, and
$|q(t)-1|\le|t-\tfrac12|$ proves continuity at $1/2$.
Thus the fundamental theorem of calculus gives
$F\in C^1([0,1])$ and $F'=q$.
Since $|q(t)-1|\le1/2$, we have
$1/2\le F'(t)\le3/2$ throughout $[0,1]$.
Moreover, $q(t)+q(1-t)=2$ for every $t\in[0,1]$.
Integrating and substituting $s=1-t$ in the second integral
gives $2\int_0^1q(t)\,dt=2$.
Consequently, $F(0)=0$ and $F(1)=1$.
The positive lower bound on $F'$ makes $F$ strictly increasing,
so its range is $[0,1]$.

Set $\Phi_F(A)=F(\lambda(A))$ for $A\in\B$.
Since both $\lambda$ and $F$ are increasing, $\Phi_F$ is increasing.
For every finite inclusion chain
$\varnothing=A_0\subseteq\cdots\subseteq A_r=[0,1]$, its
variation sum telescopes:
\[
 \sum_{i=1}^r|\Phi_F(A_i)-\Phi_F(A_{i-1})|
 =\Phi_F([0,1])-\Phi_F(\varnothing)=1.
\]
Hence $\Var(\Phi_F)=1$.
The upper bound on $F'$ also gives
$|F(s)-F(t)|\le\frac32|s-t|$ for $s,t\in[0,1]$.
For $A,B\in\B$, the identity
$\lambda(A)-\lambda(B)
=\lambda(A\setminus B)-\lambda(B\setminus A)$ implies
$|\lambda(A)-\lambda(B)|\le\lambda(A\triangle B)$.
Therefore
\[
 |\Phi_F(A)-\Phi_F(B)|
 \le\frac32|\lambda(A)-\lambda(B)|
 \le\frac32\lambda(A\triangle B),
\]
which is the required Lipschitz estimate.

It remains to rule out a bounded submodular difference.
We do this by showing that $F'$ has infinite ordinary total
variation on $[1/2,3/4]$.
For $n\ge1$, put
$a_n=(\pi/2+n\pi)^{-1/2}$ and $t_n=1/2+a_n$.
Then $a_n\downarrow0$ and $t_n\downarrow1/2$.
Choose $n_0$ such that $t_{n_0}<3/4$.
For every $n\ge n_0$,
$\sin(a_n^{-2})=\sin(\pi/2+n\pi)=(-1)^n$, so
$F'(t_n)=1+(-1)^na_n$ and
\[
 |F'(t_n)-F'(t_{n+1})|=a_n+a_{n+1}.
\]
For each integer $m>n_0$, consider the finite partition
$1/2<t_m<t_{m-1}<\cdots<t_{n_0}<3/4$, including both
endpoints. Discarding the nonnegative contributions from
the first and last subintervals gives
\[
 \Var_{[1/2,\,3/4]}(F')
 \ge\sum_{n=n_0}^{m-1}
       |F'(t_n)-F'(t_{n+1})|
 =\sum_{n=n_0}^{m-1}(a_n+a_{n+1}).
\]
For $n\ge1$, the inequality $\pi/2+n\pi\le2\pi n$ yields
$a_n\ge(2\pi n)^{-1/2}$.
Since $\sum_{n\ge n_0}n^{-1/2}$ diverges, these variation
sums are unbounded as $m\to\infty$.
Thus $\Var_{[1/2,\,3/4]}(F')=\infty$.
The interval $[1/2,3/4]$ lies strictly inside $(0,1)$,
so Lemma~\ref{lem:derivative} implies that $F$ is not a
difference of two continuous concave functions.
By the scalar criterion already proved in
Theorem~\ref{thm:main}, $\Phi_F=F\circ\lambda$ is therefore
not a difference of two bounded submodular functions.
\end{proof}

The preceding example shows that a positive bounded continuous
derivative does not suffice for bounded submodular decomposability.
We now retain the same variation obstruction while imposing a
log-Lipschitz modulus on the derivative. This gives a single
counterexample in every class $C^{1,\alpha}$, $0<\alpha<1$,
whereas $C^{1,1}$ regularity guarantees a DC representation.

\begin{theorem}\label{thm:smooth}
There exist $F\in C^1([0,1])$ and a constant $C>0$ such that
$F(0)=0$, $F(1)=1$, $1/2\le F'(t)\le3/2$ for all
$t\in[0,1]$, and
\[
 |F'(s)-F'(t)|
 \le C|s-t|\log\frac{e}{|s-t|}
 \qquad(s,t\in[0,1],\ s\ne t).
\]
Nevertheless, $F$ is not DC. Moreover, $F$ can be chosen
smooth away from one interior point and belongs simultaneously
to every $C^{1,\alpha}([0,1])$, $0<\alpha<1$.
For every atomless probability space $(X,\A,m)$,
the set function $\Phi_F(A)=F(m(A))$ is increasing, has
variation one, and satisfies
$|\Phi_F(A)-\Phi_F(B)|\le\frac32m(A\triangle B)$,
but is not a difference of two bounded submodular functions.
In contrast, every $C^{1,1}$ function on $[0,1]$ is a
difference of two continuous concave functions.
\end{theorem}

\begin{proof}
Set
\[
 Z=\sum_{n=2}^{\infty}\frac1{n(\log n)^2},
 \qquad c=\frac1{4Z},\qquad
 \ell_n=\frac{c}{n(\log n)^2},
 \qquad a_n=\frac1{n\log n}.
\]
The integral test gives $0<Z<\infty$, so
$\sum_{n=2}^{\infty}\ell_n=1/4$.
Starting with $x_2=1/2$, define $x_{n+1}=x_n+\ell_n$.
Then $x_n\uparrow x_*=3/4$.
Choose $\beta\in C_c^\infty((0,1))$ with
$0\le\beta\le1$ and $\beta(1/2)=1$, and extend it by zero
to $\R$. Define
\[
 q(t)=
 \begin{cases}
  a_n\beta\bigl((t-x_n)/\ell_n\bigr),
      &t\in[x_n,x_{n+1}]\text{ for some }n\ge2,\\
  0,&\text{otherwise}.
 \end{cases}
\]
This definition is consistent at shared endpoints because
$\beta(0)=\beta(1)=0$.
The sequence $(a_n)$ decreases to zero, so $q(t)\to0=q(x_*)$
as $t\uparrow x_*$. Since $q$ vanishes to the right of $x_*$,
it is continuous there. At every finite junction $x_n$,
the adjacent bumps vanish in a neighborhood of the junction,
because $\beta$ is supported in a compact subset of $(0,1)$.
Consequently, $q$ is continuous on $[0,1]$ and smooth away
from $x_*$.

Put $m_0=\int_0^1q(t)\,dt$ and define
\[
 F(t)=t+\frac12\int_0^t(q(s)-m_0)\,ds.
\]
Then $F\in C^1([0,1])$ and
$F'(t)=1+\frac12(q(t)-m_0)$.
Since $a_n\le a_2=1/(2\log2)<1$, we have
$0\le q\le1$ and $0\le m_0\le1$, whence
$1/2\le F'\le3/2$.
The definition gives $F(0)=0$, while
$\int_0^1(q(s)-m_0)\,ds=0$ gives $F(1)=1$.
Moreover, $F$ is smooth away from $x_*$ because $q$ is.

We next establish the modulus of continuity.
Let $\omega(d)=d\log(e/d)$ for $0<d\le1$, set $\omega(0)=0$,
and let $L=\|\beta'\|_\infty$.
The zero extension of $\beta$ is $L$-Lipschitz.
Since $Z\ge1/[2(\log2)^2]$, we have
$c\le(\log2)^2/2$, and therefore
$\ell_n\le1/(2n)$ for every $n\ge2$.
In particular, $\log(e/\ell_n)\ge\log n$, so
\[
 \frac{a_n}{\ell_n}
 =\frac{\log n}{c}
 \le\frac1c\log\frac e{\ell_n}.
\]
If $s,t\in[x_n,x_{n+1}]$ and $d=|s-t|>0$, then
$d\le\ell_n$ and
\[
 |q(s)-q(t)|
 \le L\frac{a_n}{\ell_n}d
 \le\frac Lc\,d\log\frac e{\ell_n}
 \le\frac Lc\,\omega(d).
\]
Now suppose that $s<t$ lie in distinct bump intervals,
say $s\in[x_n,x_{n+1}]$ and $t\in[x_m,x_{m+1}]$ with
$n<m$. Both $x_{n+1}$ and $x_m$ lie between $s$ and $t$,
and $q(x_{n+1})=q(x_m)=0$. Applying the preceding estimate
within each interval gives
\[
 |q(s)-q(t)|
 \le\frac Lc\bigl(\omega(x_{n+1}-s)+\omega(t-x_m)\bigr)
 \le\frac{2L}{c}\,\omega(t-s),
\]
because $\omega$ is nondecreasing on $[0,1]$.
If one point lies outside $[1/2,3/4)$, its value under $q$
is zero. When the other point lies in a bump interval,
an endpoint of that interval lies between the two points,
and the same argument applies using just that endpoint.
If both points lie outside $[1/2,3/4)$, their values are
both zero. Thus the estimate holds for all $s,t\in[0,1]$,
including the accumulation point $x_*$.
Since $F'(s)-F'(t)=\frac12(q(s)-q(t))$, the claimed
log-Lipschitz estimate follows with $C=L/c$.
For each $0<\alpha<1$, the function
$d^{1-\alpha}\log(e/d)$ is bounded on $(0,1]$,
so $\omega(d)\le C_\alpha d^\alpha$.
Hence the same function $F$ belongs to every
$C^{1,\alpha}([0,1])$ with $0<\alpha<1$.

To show that $F$ is not DC, let
$y_n=x_n+\ell_n/2$ be the midpoint of the $n$th bump interval.
The choice of $\beta$ gives
$q(x_n)=q(x_{n+1})=0$ and $q(y_n)=a_n$.
Writing $b_0=1-m_0/2$, we therefore have
$F'(x_n)=F'(x_{n+1})=b_0$ and
$F'(y_n)=b_0+a_n/2$.
For $M\ge2$, take the partition of $[1/2,3/4]$
consisting of
$x_2,y_2,x_3,y_3,\ldots,x_M,y_M,x_{M+1},3/4$.
Its variation sum yields
\[
 \begin{aligned}
 \Var_{[1/2,\,3/4]}(F')
 &\ge\sum_{n=2}^{M}
   \bigl(|F'(y_n)-F'(x_n)|
        +|F'(x_{n+1})-F'(y_n)|\bigr)\\
 &=\sum_{n=2}^{M}a_n
 =\sum_{n=2}^{M}\frac1{n\log n}.
 \end{aligned}
\]
The final sum tends to infinity by the integral test.
Since $[1/2,3/4]\subset(0,1)$,
Lemma~\ref{lem:derivative} rules out a representation of $F$
as a difference of two continuous concave functions.

Let $(X,\A,m)$ now be any atomless probability space.
The bounds on $F'$ imply that $\Phi_F=F\circ m$ is increasing
and that
\[
 |\Phi_F(A)-\Phi_F(B)|
 \le\frac32|m(A)-m(B)|
 \le\frac32m(A\triangle B).
\]
Since $\Phi_F(\varnothing)=0$ and $\Phi_F(X)=1$,
its chain variation is one.
The equivalence of (i) and (ii) in
Theorem~\ref{thm:continuous}, together with the failure of
$F$ to be DC, shows that $\Phi_F$ is not a difference of
two bounded submodular functions.

{
The final assertion is the $C^{1,1}$ consequence recorded
after Lemma~\ref{lem:derivative}.
}
\end{proof}

The preceding construction also gives profiles arbitrarily close
to the identity in $C^1$: multiplying the centered perturbation
of the derivative by any sufficiently small positive constant
preserves its infinite variation. We next show that the derivative
can simultaneously satisfy any prescribed concave modulus
$\omega$ with $\omega(d)/d\to\infty$, while retaining the
log-Lipschitz estimate.

\begin{proposition}[A prescribed modulus of continuity]
\label{prop:modulus}
Let $\omega:[0,1]\to[0,\infty)$ be continuous, increasing,
and concave, with $\omega(0)=0$ and
$\omega(d)/d\to\infty$ as $d\downarrow0$.
The function $F$ in Theorem~\ref{thm:smooth} can be chosen
to satisfy, in addition to all the conclusions of that theorem,
\[
 |F'(s)-F'(t)|\le C\omega(|s-t|)
 \qquad(s,t\in[0,1])
\]
for some constant $C>0$.
In particular, the same function has a log-Lipschitz derivative
and belongs to every $C^{1,\alpha}([0,1])$, $0<\alpha<1$.
Such profiles can be chosen arbitrarily close to the identity
in $C^1([0,1])$.
\end{proposition}

\begin{proof}
Set $\eta(0)=0$ and $\eta(d)=d\log(e/d)$ for $0<d\le1$,
and put $\widetilde\omega=\min\{\omega,\eta\}$.
Both $\omega$ and $\eta$ are continuous, increasing, and
concave, and their minimum has the same properties.
Indeed, for $a,b\in[0,1]$ and $0\le\theta\le1$, each of
$\omega((1-\theta)a+\theta b)$ and
$\eta((1-\theta)a+\theta b)$ is at least
$(1-\theta)\widetilde\omega(a)+\theta\widetilde\omega(b)$;
taking their minimum proves concavity.
Moreover, $\widetilde\omega(0)=0$,
$0\le\widetilde\omega(d)\le\eta(d)\le1$, and
\[
 \frac{\widetilde\omega(d)}d
 =\min\left\{\frac{\omega(d)}d,\log(e/d)\right\}
 \longrightarrow\infty
 \qquad(d\downarrow0).
\]
For each $k\ge1$, choose $0<\delta_k\le2^{-k-3}$ such that
$\widetilde\omega(\delta_k)/\delta_k\ge2^{2k}$, and set
$M_k=\lfloor2^{-k-2}/\delta_k\rfloor$.
Since $R_k=2^{-k-2}/\delta_k\ge2$, we have
$R_k/2\le\lfloor R_k\rfloor\le R_k$. Consequently,
\[
 2^{-k-3}\le M_k\delta_k\le2^{-k-2},
 \qquad
 M_k\widetilde\omega(\delta_k)
 \ge2^{-k-3}\frac{\widetilde\omega(\delta_k)}{\delta_k}
 \ge2^{k-3}.
\]
Thus the total width of the groups is finite, whereas the
sum of their bump heights is infinite.

Choose the function $\beta\in C_c^\infty((0,1))$ used in
Theorem~\ref{thm:smooth}, with $0\le\beta\le1$ and
$\beta(1/2)=1$, and extend it by zero to $\R$.
Let $b_1=1/2$, define $b_{k+1}=b_k+M_k\delta_k$, and put
$x_*=\lim_{k\to\infty}b_k$.
The width bounds give
$1/8\le\sum_{k\ge1}M_k\delta_k\le1/4$, so
$x_*\in[5/8,3/4]$.
For $0\le j<M_k$, write $b_{k,j}=b_k+j\delta_k$ and define
\[
 q(t)=
 \widetilde\omega(\delta_k)
 \beta\left(\frac{t-b_{k,j}}{\delta_k}\right)
 \quad\text{for }t\in[b_{k,j},b_{k,j}+\delta_k],
\]
setting $q(t)=0$ outside these intervals.
The definitions agree at shared endpoints because $\beta$
vanishes there. Each group contains finitely many intervals,
and their endpoints have no accumulation point other than
$x_*$.
Since $\delta_k\to0$ and
$\widetilde\omega(\delta_k)\to0$, the values of $q$ tend to
zero as $t\uparrow x_*$. Thus $q$ is continuous at $x_*$.
At every other junction, it vanishes in a neighborhood of
the junction because $\beta$ has compact support in $(0,1)$.
Hence $q$ is continuous on $[0,1]$, smooth away from $x_*$,
and satisfies $0\le q\le1$.

Let $L=\|\beta'\|_\infty$.
If $s,t$ belong to the same bump interval of width
$\delta_k$ and $d=|s-t|$, then the Lipschitz bound for
$\beta$ and concavity of $\widetilde\omega$ give
\[
 |q(s)-q(t)|
 \le L\frac{d}{\delta_k}\widetilde\omega(\delta_k)
 \le L\widetilde\omega(d).
\]
The second inequality follows from
$\widetilde\omega(d)\ge
(d/\delta_k)\widetilde\omega(\delta_k)$ for
$0\le d\le\delta_k$.
If $s<t$ lie in distinct bump intervals, choose the right
endpoint $u$ of the interval containing $s$ and the left
endpoint $v$ of the interval containing $t$.
Then $s\le u\le v\le t$ and $q(u)=q(v)=0$, so
\[
 |q(s)-q(t)|
 \le L\bigl(\widetilde\omega(u-s)
           +\widetilde\omega(t-v)\bigr)
 \le2L\widetilde\omega(t-s).
\]
If one point lies outside the bump intervals, its value
under $q$ is zero, and an endpoint of the other point's
interval lies between the two points. The same estimate
then follows using that endpoint. If both points lie
outside, their values are zero.
Consequently,
$|q(s)-q(t)|\le2L\widetilde\omega(|s-t|)$ for all
$s,t\in[0,1]$.

Put $m_0=\int_0^1q(s)\,ds$.
For any $0<\rho\le1/2$, define
\[
 F_\rho(t)=t+\rho\int_0^t(q(s)-m_0)\,ds.
\]
Then $F_\rho\in C^1([0,1])$, with
$F_\rho'(t)=1+\rho(q(t)-m_0)$.
Since $0\le q,m_0\le1$, we have
$1/2\le1-\rho\le F_\rho'\le1+\rho\le3/2$.
Also $F_\rho(0)=0$ and $F_\rho(1)=1$.
The modulus estimate for $q$ yields
\[
 |F_\rho'(s)-F_\rho'(t)|
 \le2\rho L\widetilde\omega(|s-t|)
 \le2\rho L\omega(|s-t|),
\]
and the same estimate holds with $\eta$ in place of $\omega$.
Thus $F_\rho'$ is log-Lipschitz.
Since $\eta(d)\le C_\alpha d^\alpha$ for every
$0<\alpha<1$, the same profile belongs to every
$C^{1,\alpha}([0,1])$. It is also smooth away from $x_*$.

It remains to verify that the variation obstruction survives.
On each bump interval in group $k$, the endpoint values of
$q$ are zero and its midpoint value is
$\widetilde\omega(\delta_k)$.
The corresponding values of $F_\rho'$ therefore rise from
$1-\rho m_0$ to
$1-\rho m_0+\rho\widetilde\omega(\delta_k)$
and return to $1-\rho m_0$.
A finite partition containing the endpoints and midpoint
of every bump in the first $K$ groups, together with
$1/2$ and $3/4$, consequently gives
\[
 \Var_{[1/2,\,3/4]}(F_\rho')
 \ge2\rho\sum_{k=1}^K
       M_k\widetilde\omega(\delta_k)
 \ge2\rho\sum_{k=1}^K2^{k-3}
 \longrightarrow\infty.
\]
Lemma~\ref{lem:derivative} shows that $F_\rho$ is not DC.
For every atomless probability space $(X,\A,m)$,
Theorem~\ref{thm:continuous} therefore rules out a bounded
submodular difference for $F_\rho\circ m$.
The endpoint normalization and derivative bounds give
monotonicity, chain variation one, and
$|F_\rho(m(A))-F_\rho(m(B))|
\le\frac32m(A\triangle B)$, as in
Theorem~\ref{thm:smooth}.

Finally, $|q-m_0|\le1$ implies
$\|F_\rho-\mathrm{id}\|_\infty\le\rho$ and
$\|F_\rho'-1\|_\infty\le\rho$.
Hence $F_\rho\to\mathrm{id}$ in $C^1([0,1])$ as
$\rho\downarrow0$, while every $\rho>0$ retains infinite
derivative variation. This proves all the assertions.
\end{proof}

The preceding constructions show that strong regularity of the
first derivative can coexist with the variation obstruction.
For the following H\"older family, the obstruction can instead
be seen directly on finite algebras. Lemma~\ref{lem:finite}
forces the oscillation of the subtracted component to become
arbitrarily large as the partitions are refined.

\begin{proposition}\label{prop:holder}
For $0<\alpha<1$, define $f_\alpha(t)=(2t-1)_+^\alpha$ on $[0,1]$ and
$\Phi_{f_\alpha}(A)=f_\alpha(\lambda(A))=(2\lambda(A)-1)_+^\alpha$ for $A\in\B$.
Then $\Phi_{f_\alpha}$ is increasing, takes values in $[0,1]$,
satisfies $\Phi_{f_\alpha}(\varnothing)=0$ and
$\Phi_{f_\alpha}([0,1])=1$, and has variation one. Moreover,
\[
 |\Phi_{f_\alpha}(A)-\Phi_{f_\alpha}(B)|
 \le2^\alpha\lambda(A\triangle B)^\alpha
 \qquad(A,B\in\B).
\]
In particular, it is continuous along increasing and decreasing
sequences of Borel sets. Nevertheless, $\Phi_{f_\alpha}$ is not
a difference of two bounded submodular functions.
\end{proposition}

\begin{proof}
The scalar profile $f_\alpha(t)=(2t-1)_+^\alpha$ is increasing
on $[0,1]$, with $f_\alpha(0)=0$ and $f_\alpha(1)=1$.
Thus $\Phi_{f_\alpha}=f_\alpha\circ\lambda$ has the asserted
range, monotonicity, and endpoint values.
Every finite inclusion chain
$\varnothing=A_0\subseteq\cdots\subseteq A_r=[0,1]$
therefore satisfies
\[
 \sum_{i=1}^r
 |\Phi_{f_\alpha}(A_i)-\Phi_{f_\alpha}(A_{i-1})|
 =\Phi_{f_\alpha}([0,1])-\Phi_{f_\alpha}(\varnothing)=1,
\]
so $\Var(\Phi_{f_\alpha})=1$.
To prove the H\"older estimate, first note that
$(u+v)^\alpha\le u^\alpha+v^\alpha$ for $u,v\ge0$.
When $u+v>0$, this follows by applying
$t^\alpha\ge t$ on $[0,1]$ to $u/(u+v)$ and $v/(u+v)$
and adding; the case $u=v=0$ is immediate.
Consequently, if $r\ge s\ge0$, then
$r^\alpha-s^\alpha\le(r-s)^\alpha$.
By symmetry, $|r^\alpha-s^\alpha|\le|r-s|^\alpha$
for all $r,s\ge0$.
Combining this with $|x_+-y_+|\le|x-y|$ gives
$|x_+^\alpha-y_+^\alpha|\le|x-y|^\alpha$ for real $x,y$.
Hence
\[
 \begin{aligned}
 |\Phi_{f_\alpha}(A)-\Phi_{f_\alpha}(B)|
 &\le 2^\alpha|\lambda(A)-\lambda(B)|^\alpha\\
 &\le 2^\alpha\lambda(A\triangle B)^\alpha.
 \end{aligned}
\]
If $A_n\uparrow A$ or $A_n\downarrow A$, continuity of
the finite measure $\lambda$ gives
$\lambda(A_n\triangle A)\to0$.
The displayed estimate then implies
$\Phi_{f_\alpha}(A_n)\to\Phi_{f_\alpha}(A)$.

Suppose now that $\Phi_{f_\alpha}=g-h$, where $g,h$ are bounded
submodular functions on $\B$.
For each integer $m\ge1$, choose a Borel partition
$I_1,\ldots,I_{2m}$ of $[0,1]$ with
$\lambda(I_i)=1/(2m)$, and define
$g^{(m)}(S)=g(\bigcup_{i\in S}I_i)$ and
$h^{(m)}(S)=h(\bigcup_{i\in S}I_i)$ for $S\subseteq[2m]$.
The map from an index set to the corresponding union of cells
preserves unions and intersections, so $g^{(m)},h^{(m)}$
are submodular. Their values are values of the original
functions, and therefore $\osc(h^{(m)})\le\osc(h)$.
Since a union of $k$ cells has measure $k/(2m)$, their
difference is a cardinality function with profile
$F_k^{(m)}=(k/m-1)_+^\alpha$, $0\le k\le2m$.
In particular,
$F_{m-1}^{(m)}=F_m^{(m)}=0$ and
$F_{m+1}^{(m)}=m^{-\alpha}$, so the positive second
difference at $k=m$ is $\delta=m^{-\alpha}$.
Lemma~\ref{lem:finite}, applied with $N=2m$ and $k=m$,
therefore gives
\[
 \osc(h)\ge\osc(h^{(m)})
 \ge\frac m2\,m^{-\alpha}
 =\frac12m^{1-\alpha}.
\]
This holds for every $m\ge1$.
Since $1-\alpha>0$, the right-hand side tends to infinity,
contradicting boundedness of $h$.
\end{proof}

The exponent range in Proposition~\ref{prop:holder} is sharp
within this family. For $\alpha\ge1$, the profile
$f_\alpha(t)=(2t-1)_+^\alpha$ is continuous and convex, since
$t\mapsto(2t-1)_+$ is convex and $r\mapsto r^\alpha$ is
convex and increasing on $[0,\infty)$.
Thus $-f_\alpha$ is continuous and concave, and
$\Phi_{f_\alpha}=0-((-f_\alpha)\circ\lambda)$ is a difference
of two bounded submodular functions.

The preceding finite-algebra argument uses a single positive
second difference to force a large oscillation in the subtracted
component. For strict majority functions, the same bound is exact.
We first establish this finite statement and then use it to
construct a product-continuous counterexample on a countable
ground set.

\begin{lemma}\label{lem:threshold}
Let $m\ge1$ be an integer, let $|J|=2m$, and let $a>0$.
If
$a\ind_{\{|S|\ge m+1\}}=g(S)-h(S)$ for every $S\subseteq J$,
where $g,h:2^J\to\R$ are submodular, then
$\osc(h)\ge am/2$.
For every $m$ and $a$, equality is attained by cardinality
functions $g,h$.
\end{lemma}

\begin{proof}
The function being decomposed has cardinality profile
$F_k=a\ind_{\{k\ge m+1\}}$, $0\le k\le2m$.
In particular, $F_{m-1}=F_m=0$ and $F_{m+1}=a$, so its
second difference at $k=m$ is $a$.
Lemma~\ref{lem:finite}, applied with $N=2m$ and $k=m$,
therefore gives $\osc(h)\ge am/2$.

To attain this bound, define
$H_k=\frac a2\min\{k,2m-k\}$ and
$G_k=H_k+a\ind_{\{k\ge m+1\}}$ for $0\le k\le2m$.
The successive increments $H_k-H_{k-1}$ are $a/2$ for
$1\le k\le m$ and $-a/2$ for $m+1\le k\le2m$.
The threshold profile has a single nonzero increment,
equal to $a$ at index $m+1$.
Consequently, $G_k-G_{k-1}$ is $a/2$ for
$1\le k\le m+1$ and $-a/2$ for $m+2\le k\le2m$,
with the latter range empty when $m=1$.
Both increment sequences are nonincreasing, so
$g(S)=G_{|S|}$ and $h(S)=H_{|S|}$ are submodular
cardinality functions. Their difference is the prescribed
strict majority function.
Finally, $H_k\ge0$, $H_0=H_{2m}=0$, and the maximum is
$H_m=am/2$. Hence $\osc(h)=am/2$, proving sharpness.
\end{proof}

The lemma separates the amplitude of a majority function
from the oscillation required in a submodular decomposition:
the former is $a$, whereas the latter is $am/2$.
We exploit this difference by choosing summable amplitudes
on blocks whose sizes grow sufficiently fast. The resulting
series is continuous, but its individual blocks force
unbounded decomposition costs.

\begin{theorem}\label{thm:countable}
Let $(B_n)_{n\ge1}$ be pairwise disjoint finite sets with
$|B_n|=2\cdot4^n$, and let $J=\bigsqcup_{n\ge1}B_n$.
Define
\[
 \Psi(A)=\sum_{n=1}^{\infty}
 2^{-n}\ind_{\{|A\cap B_n|\ge4^n+1\}},
 \qquad A\subseteq J.
\]
Then $\Psi$ is increasing, takes values in $[0,1]$,
satisfies $\Psi(\varnothing)=0$ and $\Psi(J)=1$,
and has variation one.
It is continuous in the product topology on
$2^J\cong\{0,1\}^J$, where $\{0,1\}$ is discrete,
but is not a difference of two bounded submodular functions.
\end{theorem}

\begin{proof}
The $n$th summand is nonnegative and bounded above by
$2^{-n}$. Since $\sum_{n\ge1}2^{-n}=1$, the series
converges uniformly on $2^J$ and takes values in $[0,1]$.
Every summand is increasing, vanishes at $\varnothing$,
and has value $2^{-n}$ at $J$.
Thus $\Psi$ is increasing, with $\Psi(\varnothing)=0$
and $\Psi(J)=1$.
For every finite chain
$\varnothing=A_0\subseteq\cdots\subseteq A_r=J$,
monotonicity gives
\[
 \sum_{i=1}^r|\Psi(A_i)-\Psi(A_{i-1})|
 =\Psi(J)-\Psi(\varnothing)=1.
\]
Hence $\Var(\Psi)=1$.
To verify product continuity explicitly, put
$J_N=\bigcup_{n=1}^N B_n$.
If $A,C\subseteq J$ agree on $J_N$, their first $N$
summands coincide, and therefore
\[
 |\Psi(A)-\Psi(C)|
 \le\sum_{n>N}2^{-n}=2^{-N}.
\]
For fixed $A$, the sets $C$ agreeing with $A$ on the
finite set $J_N$ form an open neighborhood of $A$
in the product topology. Since $2^{-N}\to0$,
the displayed estimate proves continuity at every
$A\subseteq J$.

Suppose, for a contradiction, that $\Psi=g-h$, where
$g,h:2^J\to\R$ are bounded and submodular.
Fix $n\ge1$, and restrict both functions to subsets
of $B_n$. These restrictions remain submodular,
because unions and intersections of subsets of $B_n$
are unchanged when viewed as subsets of $J$.
For $S\subseteq B_n$, every summand defining $\Psi(S)$
except the $n$th is zero: if $j\ne n$, then
$S\cap B_j=\varnothing$.
Hence the restricted decomposition satisfies
\[
 g(S)-h(S)
 =2^{-n}\ind_{\{|S|\ge4^n+1\}}
 \qquad(S\subseteq B_n).
\]
Lemma~\ref{lem:threshold}, with $a=2^{-n}$ and $m=4^n$,
gives
\[
 \osc(h)\ge\osc\bigl(h|_{2^{B_n}}\bigr)
 \ge\frac{2^{-n}4^n}{2}=2^{n-1}.
\]
Since this holds for every $n$, it contradicts boundedness
of $h$. Therefore $\Psi$ is not a difference of two
bounded submodular functions.
\end{proof}


{
In this example, the individual block costs are unbounded.
} We now show that even uniformly bounded block costs
can prevent a bounded submodular difference when their sum
diverges. Let $(B_n)_{n\ge1}$ be pairwise disjoint finite sets
with $|B_n|=N_n\ge2$, and let
$f_n:\{0,\ldots,N_n\}\to\R$ satisfy
$f_n(0)=0$ and $\sum_{n\ge1}\osc(f_n)<\infty$.
On $J=\bigsqcup_{n\ge1}B_n$, define
\[
 \Psi(A)=\sum_{n=1}^{\infty}f_n(|A\cap B_n|),
 \qquad A\subseteq J.
\]
Since $f_n(0)=0$, we have
$|f_n(j)|\le\osc(f_n)$ for every $j$.
Thus the series converges absolutely and uniformly.
Each summand depends on finitely many coordinates, so
$\Psi$ is continuous in the product topology.
Put
$\delta_{n,k}=
(f_n(k+1)-2f_n(k)+f_n(k-1))_+$ and
\[
 H_n^*(j)=\sum_{k=1}^{N_n-1}K_{N_n}(j,k)\delta_{n,k},
 \qquad
 c_n=\max_{0\le j\le N_n}H_n^*(j).
\]
By Theorem~\ref{thm:finite}, $c_n$ is the minimum oscillation
of the subtracted component in a submodular difference
decomposition of the $n$th block function.
The following theorem shows that these costs add exactly.

\begin{theorem}\label{thm:blocks}
The function $\Psi$ is a difference of two bounded submodular
functions if and only if $\sum_{n\ge1}c_n<\infty$.
More precisely,
\[
 \inf_{\substack{\Psi=g-h\\
                  g,h\ \text{bounded and submodular}}}
       \osc(h)
 =\sum_{n=1}^{\infty}c_n,
\]
where the infimum is $+\infty$ if no such decomposition exists.
Whenever the value is finite, the infimum is attained.
\end{theorem}

\begin{proof}
Suppose first that $\sum_{n\ge1}c_n<\infty$.
By Theorem~\ref{thm:finite}, both $H_n^*$ and
$G_n^*=f_n+H_n^*$ are concave profiles.
Moreover, $H_n^*$ is nonnegative, vanishes at $0,N_n$,
and has maximum $c_n$.
Thus, for $0\le j\le N_n$,
\[
 |f_n(j)|\le\osc(f_n),\qquad
 0\le H_n^*(j)\le c_n,\qquad
 |G_n^*(j)|\le\osc(f_n)+c_n.
\]
Consequently, the series
\[
 h(A)=\sum_{n=1}^{\infty}H_n^*(|A\cap B_n|),
 \qquad
 g(A)=\sum_{n=1}^{\infty}G_n^*(|A\cap B_n|)
\]
converge absolutely and uniformly on $2^J$ and define
bounded functions satisfying $g-h=\Psi$.
Each summand is submodular: its cardinality profile is
concave on $\{0,\ldots,N_n\}$, and the map $A\mapsto A\cap B_n$
preserves unions and intersections.
Finite partial sums are therefore submodular, and passing
to the limit in their submodular inequalities shows that
$g,h$ are submodular as well.
Since $h\ge0$ and $h(\varnothing)=0$, its minimum is zero.
For each $n$, choose $j_n$ maximizing $H_n^*$ and a set
$S_n\subseteq B_n$ of size $j_n$.
Then $A_*=\bigcup_{n\ge1}S_n$ satisfies
$h(A_*)=\sum_{n\ge1}c_n$, while $h(A)\le\sum_{n\ge1}c_n$
for every $A\subseteq J$.
Hence $\osc(h)=\sum_{n\ge1}c_n$.

We next prove the matching lower bound for an arbitrary
bounded submodular decomposition $\Psi=g-h$.
Fix $M\ge1$ and put $J_M=\bigcup_{n=1}^M B_n$.
Restrict $g,h$ to subsets of $J_M$, leaving all other
blocks empty. These restrictions remain submodular, and
$f_n(0)=0$ gives
$\Psi(A)=\sum_{n=1}^Mf_n(|A\cap B_n|)$ for
$A\subseteq J_M$.
Average the restrictions over the finite group
$\prod_{n=1}^M\Sym(B_n)$ of permutations acting separately
within the blocks. The averaged functions remain submodular
and depend only on the vector of block cardinalities.
Write their profiles as $G(j_1,\ldots,j_M)$ and
$H(j_1,\ldots,j_M)$, where $0\le j_n\le N_n$.
Since $\Psi$ is invariant under these permutations,
\[
 G(j_1,\ldots,j_M)-H(j_1,\ldots,j_M)
 =\sum_{n=1}^Mf_n(j_n).
\]
Every value of $H$ is an average of values of the original
function $h$, and hence lies between
$\inf_{A\subseteq J}h(A)$ and $\sup_{A\subseteq J}h(A)$.

Fix a coordinate $n$ and all the other coordinates.
The corresponding slices of $G$ and $H$ are concave
sequences. Indeed, fix subsets of the other blocks with
the prescribed cardinalities and apply submodularity to
two sets whose parts in $B_n$ have size $j$, intersection
size $j-1$, and union size $j+1$.
For the slice $Q_j=H(j_1,\ldots,j_{n-1},j,
j_{n+1},\ldots,j_M)$, define
$\kappa_k=2Q_k-Q_{k-1}-Q_{k+1}$.
Concavity of $Q$ gives $\kappa_k\ge0$.
The corresponding slice of $G$ equals
$Q_j+f_n(j)$ plus a constant, so its concavity also
gives
$\kappa_k\ge
f_n(k+1)-2f_n(k)+f_n(k-1)$.
Therefore $\kappa_k\ge\delta_{n,k}$.
The discrete kernel identity from
Theorem~\ref{thm:finite} now yields
\[
 \begin{aligned}
 Q_j-\left(1-\frac j{N_n}\right)Q_0
       -\frac j{N_n}Q_{N_n}
 &=\sum_{k=1}^{N_n-1}K_{N_n}(j,k)\kappa_k\\
 &\ge H_n^*(j).
 \end{aligned}
\]
This inequality holds for every choice of the other
coordinates.

For each $1\le n\le M$, choose $j_n$ with
$H_n^*(j_n)=c_n$.
Let $Z_1,\ldots,Z_M$ be independent random variables with
$\mathbb P(Z_n=N_n)=j_n/N_n$ and
$\mathbb P(Z_n=0)=1-j_n/N_n$.
Set $V_0=H(j_1,\ldots,j_M)$ and
\[
 V_k=\mathbb E\,
 H(Z_1,\ldots,Z_k,j_{k+1},\ldots,j_M)
 \qquad(1\le k\le M).
\]
For each fixed realization of $Z_1,\ldots,Z_{k-1}$,
apply the slice inequality in coordinate $k$ at $j_k$.
The endpoint interpolation is precisely the expectation
over $Z_k$, and the added term is
$H_k^*(j_k)=c_k$, independently of the other coordinates.
Taking expectations therefore gives
$V_{k-1}\ge V_k+c_k$.
Summing these inequalities yields
\[
 H(j_1,\ldots,j_M)
 -\mathbb E\,H(Z_1,\ldots,Z_M)
 \ge\sum_{n=1}^M c_n.
\]
Both terms on the left lie between $\inf h$ and $\sup h$,
so $\osc(h)\ge\sum_{n=1}^M c_n$.
Since $M$ is arbitrary and $c_n\ge0$, we conclude that
$\osc(h)\ge\sum_{n\ge1}c_n$.
In particular, a bounded decomposition can exist only when
this sum is finite. Together with the construction above,
this proves the criterion, the exact cost, and attainment.
\end{proof}

For strict majority blocks, let $N_n=2m_n$ and
$f_n(j)=a_n\ind_{\{j\ge m_n+1\}}$, where
$m_n\ge1$, $a_n\ge0$, and $\sum_{n\ge1}a_n<\infty$.
Lemma~\ref{lem:threshold} gives $c_n=a_nm_n/2$ when
$a_n>0$, and the same formula is immediate when $a_n=0$.
Thus Theorem~\ref{thm:blocks} gives
\[
 \Psi\text{ is a bounded submodular difference}
 \quad\Longleftrightarrow\quad
 \sum_{n=1}^{\infty}a_nm_n<\infty,
 \qquad
 \inf\osc(h)=\frac12\sum_{n=1}^{\infty}a_nm_n,
\]
where the infimum is over bounded submodular decompositions
$\Psi=g-h$, with the same $+\infty$ convention.
In particular, taking $a_n=2^{-n}$ and $m_n=2^n$ gives
an increasing, product-continuous function with endpoint
values zero and one, and hence variation one.
Every block has cost exactly $1/2$, but the sum of these
costs diverges. Thus the function has no bounded submodular
difference even though its individual block costs are
uniformly bounded.

\section{Monotonic decompositions of graph cuts}\label{sec:cuts}

We now study monotonic sum-decompositions of graph cut functions,
with both components normalized at the empty set.
We first establish the formula for $K_5$-free graphs.
The proof also identifies the equality conditions, which we
then use to characterize the symmetrized optimal components.
After proving necessity of the graph condition, we derive
the planar and computational consequences through linear
programming duality.

We retain the notation $G=(V,E)$, $w\ge0$, and
$W=\sum_{e\in E}w_e$ from the introduction.
Recall that the coverage function of a nonempty set
$Q\subseteq V$ is
$C_Q(S)=\ind_{\{S\cap Q\ne\varnothing\}}$.
{
It is increasing, submodular, and zero at the empty set
by Lemma~\ref{lem:submodularity}(iv).
}
Nonnegative combinations of coverage functions will provide
the required upper-bound construction.

\begin{proof}[Proof of the formula for $K_5$-free graphs]
{
Lemma~\ref{lem:known-cuts}(ii) gives
$N_G(w)\le W-\nu_w^*(G)$ for every finite simple graph,
with the explicit decomposition~\eqref{eq:optimizer}.
}

For the reverse inequality, assume that $G$ is $K_5$-free
and let $d_w=u+v$ be any normalized monotonic
sum-decomposition.
Lemma~\ref{lem:cutsym} gives an increasing submodular
function $f$ satisfying~\eqref{eq:cutidentity},
$f(\varnothing)=0$, and $f(V)=u(V)$.
{
Apply Lemma~\ref{lem:known-cuts}(i) to the decomposition
$d_w=f+(f(V\setminus\cdot)-f(V))$.
Writing $a_U$ for the multilinear coefficients of $f$,
we have $a_U=0$ for every nonclique $U$.
The $K_5$-free assumption therefore bounds the degree by four.
}

Writing $\mathbf 1$ for the all-ones vector,
the complement identity becomes
\[
 f(x)+f(\mathbf 1-x)-f(\mathbf 1)
 =\sum_{e=ij\in E}w_e(x_i+x_j-2x_ix_j).
\]
Both sides are multilinear, so their equality on
$\{0,1\}^V$ implies equality of coefficients.
For $|U|=4$, the coefficient of
$\prod_{i\in U}x_i$ on the left is $2a_U$:
the term from $f(x)$ contributes $a_U$, and the term
from $f(\mathbf 1-x)$ contributes $(-1)^4a_U=a_U$.
There are no contributions from larger sets, because
the degree is already at most four.
The right-hand side has degree at most two, so $a_U=0$
for every $|U|=4$.
Together with $f(\varnothing)=0$ and clique support,
this gives
\begin{equation}\label{eq:cubic}
 f(x)=\sum_{i\in V}a_ix_i
      +\sum_{e=ij\in E}b_ex_ix_j
      +\sum_{T\in\T}c_T\prod_{i\in T}x_i.
\end{equation}
For an edge $e=ij$, the coefficient of $x_ix_j$ in
$f(x)+f(\mathbf 1-x)$ is
$2b_e+\sum_{T\supseteq e}c_T$.
Comparing with the cut polynomial yields
\[
 2b_e+\sum_{T\supseteq e}c_T=-2w_e,
 \qquad
 b_e=-w_e-\frac12\sum_{T\supseteq e}c_T.
\]

We now use submodularity to bound the cubic coefficients.
For $e=ij$ and $S\subseteq V\setminus\{i,j\}$,
equation~\eqref{eq:cubic} gives
\[
 \Delta_{ij}f(S)
 =b_e+\sum_{\substack{T=\{i,j,k\}\in\T\\ k\in S}}c_T.
\]
Distinct triangles containing $e$ have distinct third
vertices. We may therefore choose
$S_e=\{k:\{i,j,k\}\in\T,\ c_{\{i,j,k\}}>0\}$
to retain exactly the positive terms.
Submodularity and the formula for $b_e$ then give
\[
 \begin{aligned}
 0\ge\Delta_{ij}f(S_e)
 &=-w_e-\frac12\sum_{T\supseteq e}c_T
       +\sum_{\substack{T\supseteq e\\c_T>0}}c_T\\
 &=-w_e+\frac12\sum_{T\supseteq e}|c_T|.
 \end{aligned}
\]
Hence $\sum_{T\supseteq e}|c_T|\le2w_e$ for every edge.
Define $p_T=(c_T)_+/2$.
Then $p_T\ge0$ and
$\sum_{T\supseteq e}p_T
 \le\frac12\sum_{T\supseteq e}|c_T|\le w_e$,
so $p$ is a feasible fractional triangle packing.
In particular, $\sum_Tp_T\le\nu_w^*(G)$.

It remains to relate this packing to the objective value.
Write $\deg_w(i)=\sum_{e\ni i}w_e$.
In the sum over edges incident with $i$, each triangle
containing $i$ is counted twice. Thus
$\sum_{e\ni i}b_e=-\deg_w(i)-\sum_{T\ni i}c_T$.
Using monotonicity of $f$ at the full vertex set, we obtain
\[
 \begin{aligned}
 0
 &\le f(V)-f(V\setminus\{i\})\\
 &=a_i+\sum_{e\ni i}b_e+\sum_{T\ni i}c_T
 =a_i-\deg_w(i).
 \end{aligned}
\]
Set $s_i=a_i-\deg_w(i)\ge0$.
Summing over vertices gives
$\sum_i a_i=2W+\sum_i s_i$.
Similarly, each triangle is counted three times when
the formula for $b_e$ is summed over all edges, so
$\sum_e b_e=-W-\frac32\sum_Tc_T$.
Evaluating~\eqref{eq:cubic} at $\mathbf 1$ therefore yields
$f(V)=W+\sum_i s_i-\frac12\sum_Tc_T$.
Since
$c_T=(c_T)_+-(-c_T)_+$ and $p_T=(c_T)_+/2$,
we obtain the gap identity
\begin{equation}\label{eq:gap}
 \begin{aligned}
 f(V)-\bigl(W-\nu_w^*(G)\bigr)
 &=\sum_{i\in V}s_i
   +\frac12\sum_{T\in\T}(-c_T)_+\\
 &\quad+\left(\nu_w^*(G)-\sum_{T\in\T}p_T\right).
 \end{aligned}
\end{equation}
All terms on the right are nonnegative.
Since $f(V)=u(V)$, every normalized monotonic
sum-decomposition has objective value at least
$W-\nu_w^*(G)$.
Together with the upper-bound construction, this proves
\eqref{eq:graphformula} for $K_5$-free graphs.
In particular, the pair in~\eqref{eq:optimizer} is optimal
whenever $p$ is a maximum fractional triangle packing.
\end{proof}

The gap identity~\eqref{eq:gap} also characterizes equality
in the lower bound. We call an increasing submodular function
$f:2^V\to\R$ admissible if $f(\varnothing)=0$ and
it satisfies~\eqref{eq:cutidentity}, and optimal if, in addition,
$f(V)=N_G(w)$. The next proposition establishes the bijection
in Theorem~\ref{thm:graphs} between maximum fractional triangle
packings and these symmetrized optimal components.

\begin{proposition}\label{prop:optimizers}
Let $G=(V,E)$ be a finite simple $K_5$-free graph with
nonnegative edge weights $w$.
An admissible function $f$ is optimal if and only if
$f=f_p$ in~\eqref{eq:optimizer} for a maximum fractional
triangle packing $p$.
The packing $p$ is uniquely determined by $f$.
Moreover, the symmetrization in Lemma~\ref{lem:cutsym}
sends every optimal normalized monotonic sum-decomposition
to such a function $f_p$.
\end{proposition}

\begin{proof}
Let $f$ be admissible.
The preceding coefficient argument gives its cubic
representation~\eqref{eq:cubic}, together with
\[
 b_e=-w_e-\frac12\sum_{T\supseteq e}c_T,\qquad
 \sum_{T\supseteq e}|c_T|\le2w_e
 \quad(e\in E),
\]
and $s_i=a_i-\deg_w(i)\ge0$ for every $i\in V$.
In particular, $\widehat p_T=(c_T)_+/2$ defines a feasible
fractional triangle packing, and~\eqref{eq:gap} reads
\[
 f(V)-\bigl(W-\nu_w^*(G)\bigr)
 =\sum_{i\in V}s_i
  +\frac12\sum_{T\in\T}(-c_T)_+
  +\left(\nu_w^*(G)-\sum_{T\in\T}\widehat p_T\right).
\]
If $f$ is optimal, the left-hand side is zero because
$N_G(w)=W-\nu_w^*(G)$.
Each term on the right is nonnegative, so all three vanish.
Since the sums are finite, this implies
$a_i=\deg_w(i)$ for every $i\in V$,
$c_T\ge0$ for every $T\in\T$, and
$\frac12\sum_{T\in\T}c_T=\nu_w^*(G)$.
Set $p_T=c_T/2$.
The edge inequalities above give
$\sum_{T\supseteq e}p_T\le w_e$, so $p$ is feasible;
the equality of its total weight with $\nu_w^*(G)$
makes it maximum. Moreover,
$b_e=-w_e-\sum_{T\supseteq e}p_T$.

We now identify $f$ with the function supplied by this
packing. Put $r_e=w_e-\sum_{T\supseteq e}p_T$.
For membership vectors $x\in\{0,1\}^V$, the identity
$C_Q(x)=1-\prod_{i\in Q}(1-x_i)$ expands the coverage
functions in~\eqref{eq:optimizer}.
The coefficient of $x_i$ in $f_p$ is
$\sum_{e\ni i}r_e+2\sum_{T\ni i}p_T$.
Each triangle containing $i$ contains exactly two edges
incident with $i$, so this coefficient equals $\deg_w(i)$.
For an edge $e=ij$, the coefficient of $x_ix_j$ is
$-r_e-2\sum_{T\supseteq e}p_T
=-w_e-\sum_{T\supseteq e}p_T$.
Finally, the coefficient of $\prod_{i\in T}x_i$ is
$2p_T$, and there are no other nonzero coefficients.
Consequently,
\[
 \begin{aligned}
 f_p(x)
 &=\sum_{i\in V}\deg_w(i)x_i
   -\sum_{e=ij\in E}
     \left(w_e+\sum_{T\supseteq e}p_T\right)x_ix_j\\
 &\quad+2\sum_{T\in\T}p_T\prod_{i\in T}x_i.
 \end{aligned}
\]
These are exactly the coefficients of $f$.
Uniqueness of the multilinear representation therefore
gives $f=f_p$.

Conversely, let $p$ be a maximum fractional triangle packing.
The construction in~\eqref{eq:optimizer} gives an increasing
submodular function $f_p$ that vanishes at the empty set
and satisfies~\eqref{eq:cutidentity}.
Its value at $V$ is
$f_p(V)=W-\sum_Tp_T=W-\nu_w^*(G)=N_G(w)$,
so it is optimal.
The coefficient of the monomial belonging to $T$ is
$2p_T$, and hence $f_p$ determines every coordinate of $p$.
Equivalently, for each $T\in\T$,
\[
 p_T=\frac12\sum_{R\subseteq T}(-1)^{3-|R|}f_p(R).
\]
This proves uniqueness of the packing associated with
the given function.

Finally, if $d_w=u+v$ is an optimal normalized monotonic
sum-decomposition, Lemma~\ref{lem:cutsym} produces an
admissible $f$ with $f(V)=u(V)=N_G(w)$.
The characterization above therefore gives $f=f_p$
for a uniquely determined maximum fractional triangle
packing $p$.
\end{proof}\begin{proof}[Completion of the proof of Theorem~\ref{thm:graphs}]
It remains to prove necessity of the graph condition.
For any graph containing a $K_5$ subgraph, we will construct
a nonnegative edge weighting for which
\eqref{eq:graphformula} fails.
{
Let $U=V(K_5)$. Lemma~\ref{lem:known-cuts}(iii) gives
$N_{K_5}=\max_{S\subseteq U}|S|(5-|S|)=6$.
}
To compute the packing value, note that $K_5$ has ten
triangles and each edge belongs to three of them.
Assigning weight $1/3$ to each triangle therefore gives
a feasible fractional triangle packing of total weight
$10/3$. Conversely, summing the ten edge-capacity constraints
of any feasible packing counts each triangle three times
and gives $3\sum_Tp_T\le10$.
Thus $\nu^*(K_5)=10/3$.

Now let $G$ contain a $K_5$ on vertex set $U$.
Assign weight one to the ten edges within $U$ and weight
zero to all other edges. Identifying $G[U]$ with $K_5$,
we have $d_w(S)=d_{K_5}(S\cap U)$ for every
$S\subseteq V(G)$.
If $d_w=u+v$ is a normalized monotonic sum-decomposition,
its restriction to $2^U$ is a normalized monotonic
sum-decomposition of $d_{K_5}$, because restriction
preserves submodularity and the required monotonicity.
Consequently,
$N_{K_5}\le u(U)\le u(V(G))$.
Taking the infimum over all such decompositions gives
$N_{K_5}\le N_G(w)$.
Conversely, let $d_{K_5}=u_0+v_0$ be the decomposition
{
of value six supplied by Lemma~\ref{lem:known-cuts}(iii), viewed on $2^U$.
}
Define
$u(S)=u_0(S\cap U)$ and $v(S)=v_0(S\cap U)$.
The map $S\mapsto S\cap U$ preserves unions, intersections,
and inclusion, so $u,v$ are submodular, $u$ is increasing,
and $v$ is decreasing.
Both vanish at the empty set, and
$u(S)+v(S)=d_{K_5}(S\cap U)=d_w(S)$.
Their objective value is $u(V(G))=u_0(U)=6$.
Hence $N_G(w)\le6=N_{K_5}$, and therefore $N_G(w)=6$.

Every triangle not contained in $U$ contains an edge
of weight zero. The capacity constraint on that edge,
together with nonnegativity of the packing variables,
forces the triangle's packing weight to be zero.
Thus restriction to the triangles in $U$ gives a bijection
between feasible packings for $(G,w)$ and feasible packings
for the unweighted $K_5$; the inverse extends a packing
by zero on all other triangles.
It follows that $\nu_w^*(G)=10/3$.
Since $W=10$, we obtain
\[
 N_G(w)=6<\frac{20}{3}=W-\nu_w^*(G).
\]
This nonnegative weighting violates~\eqref{eq:graphformula},
proving necessity.
Together with the formula already established for $K_5$-free
graphs and the characterization in
Proposition~\ref{prop:optimizers}, this completes the proof.
\end{proof}

Since every planar graph is $K_5$-free,
Theorem~\ref{thm:graphs} gives $N_G=|E|-\nu^*(G)$
in the unweighted case and $N_G(w)=W-\nu_w^*(G)$
for nonnegative edge weights.
This answers the planar-graph question
in~\cite[Section~8, item~2]{BGILS}, including its weighted
version. The dual of the triangle-packing program provides
a complementary description of the optimum, from which
edge-weight monotonicity follows directly.

\begin{corollary}\label{cor:monotone}
Let $G=(V,E)$ be a finite simple $K_5$-free graph, and put
\[
 P_G=\left\{y\in[0,1]^E:
       \sum_{e\subseteq T}y_e\ge1
       \text{ for every }T\in\T\right\}.
\]
For every $w\in\R_{\ge0}^E$,
\[
 N_G(w)=\max_{y\in P_G}\sum_{e\in E}w_e(1-y_e).
\]
In particular, if $0\le w'_e\le w_e$ for every $e\in E$,
then $N_G(w')\le N_G(w)$.
\end{corollary}

\begin{proof}
{
The standard fractional packing--covering duality
\cite[p.~332]{BGILS} gives
}
\[
 \nu_w^*(G)=
 \min\left\{\sum_{e\in E}w_ey_e:
       y_e\ge0\ (e\in E),\quad
       \sum_{e\subseteq T}y_e\ge1\ (T\in\T)\right\}.
\]
To restrict this minimum to a bounded feasible set, take
any dual-feasible $y$ and put $\bar y_e=\min\{y_e,1\}$.
If a triangle contains an edge with $y_e>1$, the corresponding
coordinate of $\bar y$ equals one, so its covering constraint
remains satisfied. If it contains no such edge, its constraint
is unchanged. Hence $\bar y$ is feasible and belongs to $P_G$.
Since $w_e\ge0$ and $\bar y_e\le y_e$ for every edge,
$\sum_e w_e\bar y_e\le\sum_e w_ey_e$.
The dual optimum may therefore be computed over $P_G$.
This polytope is nonempty, because it contains the all-ones
vector, and compact, because it is a closed subset of
$[0,1]^E$. In particular, the restricted minimum is attained.
Theorem~\ref{thm:graphs} now gives
\[
 \begin{aligned}
 N_G(w)
 &=W-\nu_w^*(G)\\
 &=W-\min_{y\in P_G}\sum_{e\in E}w_ey_e
 =\max_{y\in P_G}\sum_{e\in E}w_e(1-y_e).
 \end{aligned}
\]

The feasible set $P_G$ depends only on $G$.
Moreover, $1-y_e\ge0$ for every $y\in P_G$ and $e\in E$.
Thus, if $0\le w'_e\le w_e$ for all edges, then
$\sum_e w'_e(1-y_e)\le\sum_e w_e(1-y_e)$ for every
$y\in P_G$.
Taking maxima over the same feasible set proves
$N_G(w')\le N_G(w)$.
\end{proof}
For $K_5$-free graphs, the primal and dual programs also
provide certificates of optimality.
Let $p$ be a feasible fractional triangle packing,
let $y\in P_G$, and put
$r_e=w_e-\sum_{T\supseteq e}p_T\ge0$.
The construction in~\eqref{eq:optimizer} gives an upper
bound $W-\sum_Tp_T$ on $N_G(w)$, while
Corollary~\ref{cor:monotone} gives the lower bound
$\sum_e w_e(1-y_e)$.
Their difference is
\[
 \begin{aligned}
 \left(W-\sum_{T\in\T}p_T\right)
 -\sum_{e\in E}w_e(1-y_e)
 &=\sum_{e\in E}w_ey_e-\sum_{T\in\T}p_T\\
 &=\sum_{e\in E}r_ey_e
   +\sum_{T\in\T}p_T
      \left(\sum_{e\subseteq T}y_e-1\right).
 \end{aligned}
\]
Every summand is nonnegative by primal and dual feasibility.
Thus the two bounds agree if and only if
$r_ey_e=0$ for every edge and
$p_T(\sum_{e\subseteq T}y_e-1)=0$ for every triangle.
These conditions certify that $p$ and $y$ are optimal
for the packing and covering programs, respectively,
and that the associated decomposition is optimal.
Conversely, if $p,y$ are optimal, strong duality makes
their objective values equal, so the displayed gap
vanishes and each summand is zero.
{
These are the standard complementary-slackness conditions~\cite[Chapter~7]{Schrijver}.
}

For rational edge weights, this description yields a
polynomial-time construction in the ordinary bit model.
The triangles can be enumerated by checking all vertex
triples, and the packing program has
$|\T|\le\binom{|V|}{3}$ variables and $|E|$ capacity
constraints, in addition to nonnegativity constraints.
The restricted dual has $|E|$ variables, $|\T|$ covering
constraints, and the bounds $0\le y_e\le1$.
Both programs have rational data of polynomial encoding
length. {
Rational linear programming therefore produces
optimal rational solutions $p,y$ in polynomial time,
with polynomial encoding length~\cite[Chapters~10 and~13]{Schrijver}.
}
Formula~\eqref{eq:optimizer} represents $f_p$ using at most
$|E|+|\T|$ coverage terms, and complementation gives $g_p$.
These representations have polynomial size and permit
exact polynomial-time evaluation at any specified subset
of $V$; they do not require listing all function values.
Primal and dual feasibility, together with equality of
their objective values, form a polynomial-size certificate
that can be checked using rational arithmetic in
polynomial time.

We now pass from the symmetrized optimal components to
all optimal pairs. Symmetrization leaves a perturbation
that is invariant under complementation.
On a $K_4$-free graph, clique support and this invariance
force the perturbation to be quadratic.
The remaining monotonicity and submodularity conditions
then give the incidence and capacity constraints in
Theorem~\ref{thm:allopt}.

\begin{proof}[Proof of Theorem~\ref{thm:allopt}]
Let $d_w=u+v$ be an optimal normalized monotonic
sum-decomposition on a $K_4$-free graph.
By Lemma~\ref{lem:cutsym} and
Proposition~\ref{prop:optimizers}, its symmetrization is
$f_p$ for a unique maximum fractional triangle packing $p$.
Since $v(V)=-u(V)$ and
$v(V\setminus S)=d_w(S)-u(V\setminus S)$, the
symmetrization formula gives
\[
 2f_p(S)=u(S)+d_w(S)-u(V\setminus S)+u(V).
\]
Writing the same identity for $V\setminus S$ and
subtracting yields
$f_p(S)-f_p(V\setminus S)=u(S)-u(V\setminus S)$.
Therefore $h=u-f_p$ satisfies
$h(S)=h(V\setminus S)$ for every $S\subseteq V$.
Also, $h(\varnothing)=0$ by normalization and
$h(V)=0$ because symmetrization preserves the objective.

We next determine the possible coefficients of $h$.
{
Lemma~\ref{lem:known-cuts}(i) implies that $u$ has
coefficients supported only on cliques.
}
The same is true of $f_p$, since the monomials in
each coverage function $C_e$ or $C_T$ are supported
on subsets of the corresponding edge or triangle.
Thus $h=u-f_p$ also has clique support.
This conclusion does not require $h$ to be submodular.
Since $G$ is $K_4$-free, $h$ has degree at most three.
In the identity $h(x)=h(\mathbf 1-x)$, each cubic
coefficient $c$ on the left has coefficient $-c$
on the right, with no contributions from higher-degree
terms. Hence every cubic coefficient is zero.
As $h(\varnothing)=0$, we may write
\[
 h(x)=\sum_{i\in V}\alpha_ix_i
      +\sum_{e=ij\in E}z_ex_ix_j.
\]

The coverage representation gives
$f_p(\{i\})=\sum_{e\ni i}r_e+2\sum_{T\ni i}p_T
=\deg_w(i)$.
Indeed, each triangle containing $i$ contributes to
exactly two of the residual edge weights in this sum.
It also gives $f_p(V)-f_p(V\setminus\{i\})=0$:
every edge or triangle still meets $V\setminus\{i\}$
after a single vertex is removed.
Since $h(\{i\})=\alpha_i$ and
$h(V\setminus\{i\})=h(\{i\})$, the prescribed
monotonicities imply
\[
 \begin{aligned}
 0&\ge v(\{i\})-v(\varnothing)
    =d_w(\{i\})-f_p(\{i\})-\alpha_i
    =-\alpha_i,\\
 0&\le u(V)-u(V\setminus\{i\})
    =-\alpha_i.
 \end{aligned}
\]
Thus $\alpha_i=0$ for every $i$.
Comparing linear coefficients in
$h(x)=h(\mathbf 1-x)$ gives
$2\alpha_i+\sum_{e\ni i}z_e=0$, and hence $Bz=0$.
Summing over vertices, with each edge counted twice,
also gives $\sum_{e\in E}z_e=0$.
We have therefore obtained the stated quadratic form
of $u-f_p$.

To determine the remaining constraints on $z$, fix
an edge $e=ij$ and put $P_e=\sum_{T\supseteq e}p_T$,
so that $r_e=w_e-P_e$.
The expansion of $f_p$ gives, for
$S\subseteq V\setminus\{i,j\}$,
\[
 \Delta_{ij}u(S)
 =-w_e-P_e
  +2\sum_{\substack{T=\{i,j,k\}\in\T\\k\in S}}p_T
  +z_e.
\]
Since $p_T\ge0$, the minimum over $S$ is
$-w_e-P_e+z_e$, attained at $S=\varnothing$.
The maximum is $-w_e+P_e+z_e$, attained by taking
$S=V\setminus\{i,j\}$.
Moreover, $\Delta_{ij}d_w(S)=-2w_e$, so submodularity
of both $u$ and $v=d_w-u$ is equivalent, on this edge,
to $-2w_e\le\Delta_{ij}u(S)\le0$ for every $S$.
Using the two attained extrema, these inequalities become
\[
 -w_e-P_e+z_e\ge-2w_e,
 \qquad
 -w_e+P_e+z_e\le0,
\]
or equivalently $-r_e\le z_e\le r_e$.
This proves necessity of all the stated conditions.

Conversely, let $p$ be a maximum fractional triangle
packing and let $z\in\R^E$ satisfy $Bz=0$ and
$|z_e|\le r_e$ for every edge.
Define
$h(S)=\sum_{e\subseteq S}z_e$,
$u=f_p+h$, and $v=g_p-h$.
Then $u+v=d_w$, and both components vanish at the
empty set.
The preceding second-difference calculation and
$|z_e|\le r_e$ give
$-2w_e\le\Delta_{ij}u(S)\le0$ on every edge.
Thus $\Delta_{ij}v(S)=-2w_e-\Delta_{ij}u(S)\le0$
there as well.
On nonedges, both second differences vanish because
$f_p$ and $h$ have clique support and
$\Delta_{ij}d_w(S)=0$.
The second-difference characterization of submodularity
therefore shows that $u,v$ are submodular.

It remains to verify monotonicity.
The coverage identities above and $Bz=0$ give
\[
 u(V)-u(V\setminus\{i\})
 =\sum_{e\ni i}z_e=0.
\]
Also, $h(\{i\})=0$ and
$f_p(\{i\})=\deg_w(i)=d_w(\{i\})$, so
$v(\{i\})-v(\varnothing)=0$.
For $S\subseteq V\setminus\{i\}$, diminishing returns
for the two submodular functions now yields
\[
 \begin{aligned}
 u(S\cup\{i\})-u(S)
 &\ge u(V)-u(V\setminus\{i\})=0,\\
 v(S\cup\{i\})-v(S)
 &\le v(\{i\})-v(\varnothing)=0.
 \end{aligned}
\]
Hence $u$ is increasing and $v$ is decreasing.
Since $Bz=0$ implies $\sum_ez_e=0$, we have
$u(V)=f_p(V)=N_G(w)$.
Thus the constructed pair is an optimal normalized
monotonic sum-decomposition.

Finally, the quadratic perturbation satisfies
\[
 h(\mathbf 1-x)
 =h(x)+\sum_{e\in E}z_e
       -\sum_{i\in V}(Bz)_ix_i
 =h(x).
\]
Together with $h(V)=0$, this shows that the perturbation
cancels under the symmetrization of
Lemma~\ref{lem:cutsym}.
Indeed, the symmetrized increasing component is
\[
 \frac{u(S)+v(V\setminus S)-v(V)}2=f_p(S).
\]
Proposition~\ref{prop:optimizers} therefore recovers $p$
uniquely from the pair $(u,v)$.
Once $p$ is fixed, uniqueness of the multilinear
representation recovers $z$ from the quadratic
coefficients of $u-f_p$.
This proves uniqueness of the parametrization.
\end{proof}
Theorem~\ref{thm:allopt} also gives a compact description of
the entire polytope of optimal decompositions.
Let $A$ be the edge--triangle incidence matrix, with
$A_{e,T}=1$ if $e\subseteq T$ and $A_{e,T}=0$ otherwise,
and let $B$ be the unsigned vertex-edge incidence matrix.
For fixed $w$, the map $(p,z)\mapsto(u,v)$ in that theorem
is affine and gives a bijection from
\[
 \mathcal P_w=
 \left\{(p,z):
 \begin{array}{l}
 p\ge0,\quad Ap\le w,\quad
 \mathbf 1^\top p=\nu_w^*(G),\\
 Bz=0,\quad -(w-Ap)\le z\le w-Ap
 \end{array}\right\}
\]
onto the polytope of optimal normalized monotonic
sum-decompositions. Here all vector inequalities are
coordinatewise, and $\mathbf 1$ is indexed by $\T$.
The two polytopes are therefore affinely isomorphic.
The parameter description uses $|\T|+|E|$ variables and
polynomially many constraints.
For rational $w$, the triangle-packing program computes
the rational value $\nu_w^*(G)$ in polynomial time,
and this value has polynomial encoding length.
Thus $\mathcal P_w$ has a description of encoding length
polynomial in the graph size and the binary encoding
length of $w$.

Fixing the packing isolates the freedom that remains after
symmetrization. Its dimension depends only on the edges
with positive residual capacity, leading to the following
uniqueness criterion.

\begin{corollary}\label{cor:uniqueness}
Let $G=(V,E)$ be a finite simple $K_4$-free graph with
nonnegative weights $w$, and let $p$ be a maximum fractional
triangle packing. Put
$r_e=w_e-\sum_{T\supseteq e}p_T$ and
$H_p=(V,\{e\in E:r_e>0\})$.
Let $b(H_p)$ denote the number of bipartite connected
components of $H_p$, including isolated vertices.
The affine dimension over $\R$ of the family of optimal
pairs with symmetrization $f_p$ is
\[
 |E(H_p)|-|V|+b(H_p).
\]
There is a unique optimal normalized monotonic
sum-decomposition if and only if the maximum fractional
triangle packing $p$ is unique and every connected
component of $H_p$ is either a tree or a connected graph
with exactly one cycle, that cycle being odd.
\end{corollary}

\begin{proof}
By Theorem~\ref{thm:allopt}, fixing the symmetrized component
$f_p$ fixes the packing $p$, and the remaining parameters
satisfy $Bz=0$ and $|z_e|\le r_e$.
If $r_e=0$, then $z_e=0$.
After deleting these forced zero coordinates, the parameter
set becomes
\[
 \left\{z\in\ker B_{H_p}:
             |z_e|\le r_e\quad(e\in E(H_p))\right\},
\]
where $B_{H_p}$ is the unsigned incidence matrix of $H_p$,
with all vertices in $V$ retained.
If $E(H_p)\ne\varnothing$, every remaining $r_e$ is
positive, so the box contains a neighborhood of zero
relative to $\ker B_{H_p}$.
Its affine hull is therefore the whole kernel.
If $E(H_p)=\varnothing$, both the parameter set and the
kernel consist only of the zero vector, giving the same
conclusion. Since the affine parametrization is injective,
the family of optimal pairs has dimension
$\dim_{\R}\ker B_{H_p}$.

{
Lemma~\ref{lem:incidence-rank} gives
$\operatorname{rank}_{\R}B_{H_p}=|V|-b(H_p)$,
with isolated vertices counted as bipartite components.
}
Rank--nullity now gives
$\dim_{\R}\ker B_{H_p}=|E(H_p)|-|V|+b(H_p)$,
as claimed.

For a connected bipartite component with $n$ vertices
and $m$ edges, the kernel dimension is $m-n+1$.
It is zero exactly when $m=n-1$, which, by connectedness,
is equivalent to the component being a tree.
For a connected nonbipartite component, the kernel
dimension is $m-n$. It is zero exactly when $m=n$.
A spanning tree then leaves exactly one additional edge,
so the component has a unique cycle; that cycle must be
odd because the component is nonbipartite.
It follows that the family associated with a fixed
packing consists of a single pair exactly under the
stated component condition.
Finally, $z=0$ is feasible for every maximum packing,
and distinct maximum packings give distinct symmetrized
components by Proposition~\ref{prop:optimizers}.
Thus the entire optimal decomposition is unique if and
only if the maximum packing is unique and its associated
family has dimension zero.
\end{proof}

For example, on a cycle of length four there are no
triangles, so $p=0$ is the unique packing and $r=w$.
If all edge weights are positive, the vectors
$z=(t,-t,t,-t)$ in cyclic edge order satisfy $Bz=0$ and
$|z_e|\le r_e$ whenever $|t|\le\min_{e\in E}w_e$.
They give a nontrivial interval of optimal decompositions
with the same symmetrization.
Thus uniqueness of the maximum packing alone does not
imply uniqueness of the original pair.

The $K_4$-free assumption is also essential to the quadratic
description in Theorem~\ref{thm:allopt}.
When $K_4$ subgraphs are present, a perturbation invariant
under complementation can contain quartic terms.
To see this explicitly, take $G=K_4$ on vertices
$\{1,2,3,4\}$, give edges $12,34$ weight two, and give
the other four edges weight one.
Set $p_T=1/2$ for each of the four triangles.
Every edge belongs to two triangles, so this packing is
feasible. Each triangle contains exactly two weight-one
edges; summing the four corresponding capacity constraints
therefore gives $2\sum_Tp_T\le4$.
Our packing attains this bound, so
$\nu_w^*(G)=2$.
Since $W=8$ and $K_4$ is $K_5$-free,
Theorem~\ref{thm:graphs} gives $N_G(w)=6$.
The residual capacities are
$r_{12}=r_{34}=1$ and $r_e=0$ on the other edges.

For $|t|\le1$, define
\[
 h_t(S)=t\ind_{\{S=\{1,2\}\text{ or }S=\{3,4\}\}},
 \qquad
 u_t=f_p+h_t,\qquad v_t=g_p-h_t.
\]
We verify that $(u_t,v_t)$ is an optimal normalized
monotonic sum-decomposition.
The coverage formula gives $f_p(\varnothing)=0$,
$f_p(S)=4$ on singletons, $f_p(S)=5$ on the pairs
$\{1,2\},\{3,4\}$, and $f_p(S)=6$ on all other pairs,
triples, and the full vertex set.
For an edge $e=ij$, write
$V\setminus\{i,j\}=\{a,b\}$ and list second differences
in the order
$S=\varnothing,\{a\},\{b\},\{a,b\}$.
If $e$ is $12$ or $34$, these lists are
\[
 \Delta_{ij} f_p=(-3,-2,-2,-1),
 \qquad
 \Delta_{ij} h_t=(t,0,0,t).
\]
For any weight-one edge, they are
\[
 \Delta_{ij} f_p=(-2,-1,-1,0),
 \qquad
 \Delta_{ij} h_t=(0,-t,-t,0).
\]
Hence $-2w_e\le\Delta_{ij}u_t(S)\le0$ for every edge and
every admissible $S$ whenever $|t|\le1$.
Since $u_t+v_t=d_w$ and $\Delta_{ij}d_w(S)=-2w_e$,
we also have
$\Delta_{ij}v_t(S)=-2w_e-\Delta_{ij}u_t(S)\le0$.
All vertex pairs are edges of $K_4$, so these inequalities
prove submodularity of both components.

The perturbation vanishes on the empty set, singletons,
triples, and the full vertex set.
Thus $u_t(\varnothing)=v_t(\varnothing)=0$, and for
every vertex $i$,
$u_t(V)-u_t(V\setminus\{i\})=6-6=0$.
Also,
$v_t(\{i\})=g_p(\{i\})
=f_p(V\setminus\{i\})-f_p(V)=0$.
Diminishing returns therefore gives nonnegative marginals
for $u_t$ and nonpositive marginals for $v_t$.
Hence $u_t$ is increasing and $v_t$ is decreasing.
Its objective value is $u_t(V)=f_p(V)=6=N_G(w)$,
so the pair is optimal.
Moreover, $h_t(S)=h_t(V\setminus S)$ and $h_t(V)=0$,
which imply that the symmetrization of this pair is $f_p$.

Finally, the multilinear representation of the perturbation is
\[
 \begin{aligned}
 h_t(x)
 &=t\bigl[x_1x_2(1-x_3)(1-x_4)
          +x_3x_4(1-x_1)(1-x_2)\bigr]\\
 &=t\left(x_1x_2+x_3x_4
       -\sum_{\substack{T\subseteq V\\|T|=3}}
          \prod_{i\in T}x_i
       +2x_1x_2x_3x_4\right).
 \end{aligned}
\]
Its quartic coefficient is $2t$, so for $t\ne0$ it cannot
be represented by a quadratic polynomial.
Indeed, a perturbation satisfying the quadratic constraints
of Theorem~\ref{thm:allopt} for this packing would be supported
on the matching $\{12,34\}$, and $Bz=0$ would force
$z_{12}=z_{34}=0$.
The nonzero perturbations above therefore show directly
why that parametrization does not extend to all
$K_5$-free graphs.

\section*{Declaration of generative AI and AI-assisted technologies}
The authors used a large language model during the preparation of this manuscript
for assistance with mathematical exploration, checking arguments,
organization and language editing. The authors are responsible for
the mathematical statements, proofs, references and final text.

\bigskip
{\small
\noindent\textsc{Guoqing Wang}\par
\noindent\textsc{School of Mathematics, Tiangong University}\par
\noindent\emph{Email address:} \texttt{gqwang1979@aliyun.com}\par
\medskip
\noindent\textsc{Linlin Wang}\par
\noindent\textsc{School of Mathematics, China University of Mining and Technology}\par
\noindent\emph{Email address:} \texttt{5400@cumt.edu.cn}\par
\medskip
\noindent\textsc{Jun Yue}\par
\noindent\textsc{School of Mathematics, Tiangong University}\par
\noindent\emph{Email address:} \texttt{yuejun06@126.com}\par
}

\end{document}